\documentclass[aap]{imsart}

\RequirePackage{amsthm,amsmath,amsfonts,amssymb}
\RequirePackage[numbers]{natbib}   
\usepackage{xcolor}
\usepackage{float}
\usepackage{pgfplots}
\usepackage{listings}
\usepackage{longtable}
\usepackage{mathrsfs}
\usepackage{dsfont}
\usepackage[pdfpagelabels,hyperindex]{hyperref} 
\RequirePackage[colorlinks,citecolor=blue,urlcolor=blue]{hyperref} 
\usepackage{graphicx}
\usepackage{anyfontsize}
\usepackage{bookmark}

\startlocaldefs
\numberwithin{equation}{section}  
\theoremstyle{plain}
\newtheorem{theorem}{Theorem}[section]
\newtheorem{lemma}[theorem]{Lemma}
\newtheorem{proposition}[theorem]{Proposition}

\newtheorem{question}[theorem]{Question}
\newtheorem{claim}[theorem]{Claim} 
\theoremstyle{definition}
\newtheorem{definition}[theorem]{Definition}
\newtheorem{remark}[theorem]{Remark}

\newcommand*{\rom}[1]{\expandafter\@slowromancap\romannumeral #1@}
 \colorlet{lgray}{white!80!black}

\endlocaldefs
\colorlet{shadecolor}{gray!20}
\pgfplotsset{compat=1.9}
\usetikzlibrary{shapes.multipart}
\usetikzlibrary{patterns}
\usetikzlibrary{shapes.multipart}
\usetikzlibrary{arrows}
\usetikzlibrary{decorations.markings}
\usepgflibrary{decorations.shapes}
\usetikzlibrary{decorations.shapes}
\usepgflibrary{shapes.symbols}
\usetikzlibrary{shapes.symbols}
\usetikzlibrary{decorations.pathreplacing}
\tikzstyle{fleche}=[>=stealth', postaction={decorate}, thick]
\tikzstyle{axis}=[->, >=stealth', thick, gray]
\tikzstyle{paths}=[>->, >=stealth', thick]
\tikzstyle{path}=[->, >=stealth', thick]
\tikzstyle{grille}=[dotted, gray] 
\DeclareFontFamily{U}{mathx}{}
\DeclareFontShape{U}{mathx}{m}{n}{<-> mathx10}{}
\DeclareSymbolFont{mathx}{U}{mathx}{m}{n}
\DeclareMathAccent{\widehat}{0}{mathx}{"70}
\DeclareMathAccent{\widecheck}{0}{mathx}{"71}
\def \be {\begin{equation}}
\def \ee  {\end{equation}}
\def \lb {\left(}
\def \rb {\right)} 
\def \le {\left[}
\def \re {\right]}
\def \EE {\mathbb{E}}

\def \de {\delta}

\def \lnor {\left\lVert}
\def \rnor {\right\rVert_{TV}} 
\def \RR {\mathbb{R}}
\def \R {\mathbb{R}}
\def \ep {\varepsilon}
\def \NN {\mathbb{N}_0}
\def \d {\mathsf{d}}  
\def \e {\mathfrak{e}}
\def \f {\mathfrak{f}}
\def \eee{\mathbf{e}}
\def \fff{\mathbf{f}} 
\def \aa {\mathbf{a}}
\def \bb {\mathbf{b}}
\def \cc {\mathbf{c}}
\def \dd {\mathbf{d}}
\def \ZZ {\mathbb{Z}}
\def \RR {\mathbb{R}}
\def \CC {\mathbb{C}}
\def \th {\theta}
\def \k {\kappa} 
\def \up {\upsilon}
\def \G {\Gamma}

\def \one {\mathds{1}}
\def \i {\mathtt{i}}
 
\def \dtv {\d_{TV}}
\def \lv {\left\lvert}
\def \rv {\right\lvert}

\def \la {\lambda}

\DeclareMathOperator{\ber}{Ber} 
\DeclareMathOperator{\const}{Const}  

\DeclareMathOperator{\card}{card}

\endlocaldefs

\begin{document}

\begin{frontmatter} 
\title{Limits of open ASEP stationary measures near a boundary} 
\runtitle{Open ASEP stationary measures near a boundary} 

\begin{aug} 
\author {\fnms{Zongrui}~\snm{Yang} \ead[label=e1]{zy2417@columbia.edu}} 
\orcid{0000-0003-2551-7965} 
\address {Department of Mathematics,
                Columbia University \printead[presep={,\ }]{e1}}
\end{aug}

\begin{abstract}
Consider the stationary measure of open asymmetric simple exclusion process (ASEP) on the lattice $\{1,\dots,n\}$. Taking $n$ to infinity while fixing the jump rates, this measure converges to a measure on the semi-infinite lattice. In the high and low density phases,  we characterize the limiting measure and provide bounds on the convergence rates in total variation distance.  Our approach involves bounding the total variation distance using generating functions, which are further estimated through a subtle analysis of the atom masses of Askey--Wilson signed measures.
\end{abstract}



\end{frontmatter}

\section{Introduction}
\subsection{Preface} 
The open asymmetric simple exclusion process (ASEP) serves as a fundamental model for nonequilibrium systems with open boundaries and for Kardar–Parisi–Zhang (KPZ) universality.  Over the past five decades, extensive studies have been dedicated to understand the stationary measure of open ASEP, encompassing a wide range of its asymptotic and limiting behaviors, including the particle densities  \cite{derrida1993exact,schutz1993phase,mallick1997finite,essler1996representations,sasamoto2000density,derrida2002exact,derrida2003exact,uchiyama2004asymmetric,bryc2019limit,wang2023askey},  limit fluctuations \cite{derrida2004asymmetric,bryc2019limit,wang2023askey}, large deviations \cite{derrida2002exact,derrida2003exact,bryc2017asymmetric} and open KPZ limits \cite{corwin2021stationary,bryc2021markov,barraquand2021steady,bryc2021markov2}. See survey papers \cite{blythe2007nonequilibrium,corwin2022some} and more references therein. A significant portion of these studies stems from the matrix product ansatz (MPA) method introduced in the seminal work \cite{derrida1993exact}. This method is notably related to the Askey--Wilson polynomials \cite{uchiyama2004asymmetric,corteel2011tableaux} and processes \cite{bryc2017asymmetric}.

We are interested in a straightforward limit of the open ASEP stationary measures near the boundary. Specifically, we will consider the stationary measure on the lattice $\{1,\dots,n\}$. We fix all the parameters $q,\alpha,\beta,\gamma,\delta$ of the model and take the system size $n$ to infinity. It is known from \cite{liggett1975ergodic,sasamoto2012combinatorics} that such sequences  weakly converge, and that the limiting probability measures on $\{0,1\}^{\ZZ_+}$   are stationary measures of certain ASEP systems on the semi-infinite lattice $\ZZ_+$  
 with parameters $q,\alpha,\gamma$. 
We mention that the stationary measures of this semi-infinite ASEP are not unique and are parameterized by the limiting densities at infinity. 
  The aforementioned limits from finite lattices $\{1,\dots,n\}$ to $\ZZ_+$  were first studied by Liggett \cite{liggett1975ergodic} assuming the so-called Liggett's condition, see for example \cite[Theorem 1.8 and Theorem 3.10]{liggett1975ergodic}. Later in Grosskinsky \cite[Theorem 3.2]{grosskinsky2004phase} and Sasamoto--Williams \cite{sasamoto2012combinatorics}, a matrix product ansatz was developed to characterize the limiting measures, enabling the studies of their large deviations in Duhart \cite{de2015large} and Duhart--Mörters--Zimmer \cite{duhart2018semi}. The limiting measures were further characterized in Bryc--Weso\l owski \cite[Theorem 12]{bryc2017asymmetric} in terms of the Askey--Wilson processes, in the `fan region' part of the phase diagram.

In this paper we achieve two main goals. Firstly, using the Askey--Wilson signed measures introduced in a recent work Wang--Weso{\l}owski--Yang \cite{wang2023askey}, we characterize the limiting probability measures on $\{0,1\}^{\ZZ_+}$ in the `shock region' part of the phase diagram. This complements the characterization Bryc--Weso\l owski \cite[Theorem 12]{bryc2017asymmetric} in the  fan region and could serve as a useful tool for further studies of the asymptotics. Secondly, we investigate the following natural question:
\begin{question}
    At which scale of the (leftmost) sublattice does this convergence occur in total variation distance? To be specific, when comparing the open ASEP stationary measure on $\{0,1\}^{n}$ with the limiting measure on $\{0,1\}^{\ZZ_+}$ by measuring the total variation distance of their marginal distributions on the sublattice $\{1,\dots,m_n\}$,  under which growth rate of the sequence $m_n$, $n=1,2,\dots$ does this total variation distance converge to zero? 
\end{question}

In the case $\gamma=\delta=0$ and within the low density phase of the shock region, a recent work by Nestoridi and Schmid \cite[Theorem 1.4]{nestoridi2023approximating} provides a growth rate. In this paper, we contribute another (partial) answer to this question: In both the high and low-density phases, the convergence occurs in total variation distance on the leftmost sublattice with a scale of $n/\log n$. We note that in the high density phase, the limiting measures on $\{0,1\}^{\mathbb{Z}_+}$ are in general no longer product Bernoulli measures (as in the low density phase), necessitating a different method.

Our approach involves bounding the total variation distance between two probability measures on $\{0,1\}^m$ by the values of their joint generating functions at specific points. To bound the generating functions, subtle estimations on the total variations of certain Askey–Wilson signed measures are necessary, which are derived through careful analysis of the masses of all the atoms.

\subsection{Model and results}\label{subsec:model and results}
The open ASEP is a continuous-time irreducible Markov process on the state space $\{0,1\}^n$ with parameters
\be\label{eq:conditions open ASEP}
\alpha,\beta>0,\quad \gamma,\de\geq 0,\quad 0\leq q<1,
\ee  
which models the evolution of particles on the lattice $\left\{1,\dots,n\right\}$. In the bulk of the system, particles move at random to the left with rate $q$ and to the right with rate $1$.
At the left boundary, particles enter at random with rate $\alpha$ and exit at random with rate $\gamma$. At the right boundary, particles enter at random with rate $\delta$ and exit at random with rate $\beta$. Any move of a particle is prohibited  if the target site is already occupied. See Figure \ref{fig:openASEP} for an illustration.  
 
\begin{figure}[ht]
\centering
 \begin{tikzpicture}[scale=0.8]
\draw[thick] (-1.2, 0) circle(1.2);
\draw (-1.2,0) node{reservoir};
\draw[thick] (0, 0) -- (12, 0);
\foreach \x in {1, ..., 12} {
	\draw[gray] (\x, 0.15) -- (\x, -0.15);
}
\draw[thick] (13.2,0) circle(1.2);
\draw(13.2,0) node{reservoir};
\fill[thick] (2, 0) circle(0.2);
\fill[thick] (5, 0) circle(0.2);
\fill[thick] (7, 0) circle(0.2);
\fill[thick] (10, 0) circle(0.2);
\draw[thick, ->] (2, 0.3)  to[bend left] node[midway, above]{$1$} (3, 0.3);
\draw[thick, ->] (5, 0.3)  to[bend right] node[midway, above]{$q$} (4, 0.3);
\draw[thick, ->] (7, 0.3) to[bend left] node[midway, above]{$1$} (8, 0.3);
\draw[thick, ->] (10, 0.3) to[bend left] node[midway, above]{$1$} (11, 0.3);
\draw[thick, ->] (10, 0.3) to[bend right] node[midway, above]{$q$} (9, 0.3);
\draw[thick, ->] (-0.1, 0.5) to[bend left] node[midway, above]{$\alpha$} (0.9, 0.4);
\draw[thick, <-] (0, -0.5) to[bend right] node[midway, below]{$\gamma$} (0.9, -0.4);
\draw[thick, ->] (12, -0.4) to[bend left] node[midway, below]{$\delta$} (11, -0.5);
\draw[thick, <-] (12, 0.4) to[bend right] node[midway, above]{$\beta$} (11.1, 0.5);
\node[gray] at (1,-0.4) {\tiny $1$};\node[gray] at (2,-0.4) {\tiny $2$};\node[gray] at (3,-0.4) {\tiny $3$};\node[gray] at (4,-0.4) {\tiny $4$};\node[gray] at (5,-0.4) {\tiny\dots};\node[gray] at (11,-0.4) {\tiny $n$};
\end{tikzpicture} 
\caption{Jump rates in the open ASEP.}
\label{fig:openASEP}
\end{figure}
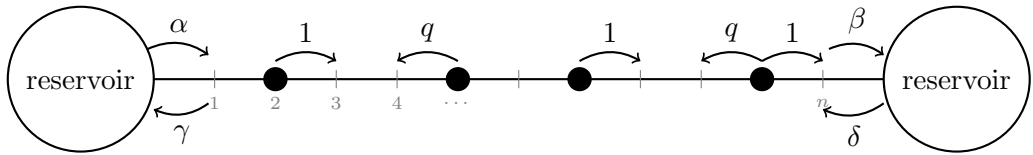 
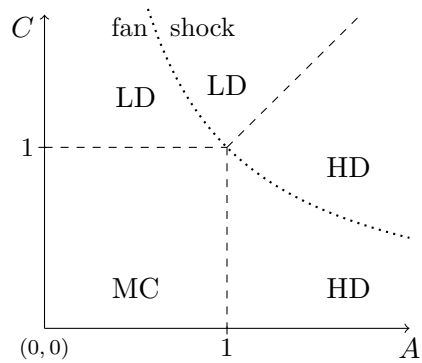
\begin{figure}[ht]
    \centering
    \begin{tikzpicture}[scale=0.8]
 \draw[scale = 1,domain=6.7:11,smooth,variable=\x,dotted,thick] plot ({\x},{1/((\x-7)*1/3+2/3)*3+5});
 \draw[->] (5,5) to (5,10.2);
 \draw[->] (5.,5) to (11,5);
   \draw[dashed] (5,8) to (8,8);
  \draw[dashed] (8,8) to (8,5);
   \draw[dashed] (8,8) to (10.2,10.2);
   \node [left] at (5,8) {\small$1$};
   \node[below] at (8,5) {\small $1$};
     \node [below] at (11,5) {$A$};
   \node [left] at (5,10) {$C$};
 \draw[dashed] (8,4.9) to (8,5.1);
  \draw[dashed] (4.9,8) to (5.1,8);
 \node [below] at (5,5) {\scriptsize$(0,0)$};
    \node [above] at (6.5,8.5) {LD};\node [above] at (8,8.7) {LD}; 
    \node [below] at (10,6) {HD};  \node [below] at (10,8) {HD};  
 \node [below] at (6.5,6) {MC}; 
 \node at (6.4,10) {\small{fan}};
 \node at (7.6,10) {\small{shock}};
 \end{tikzpicture} 
    \caption{Phase diagrams for the open ASEP stationary measures. LD, HD, MC respectively
stand for the low density, high density and maximal current phases.}
    \label{fig:phase}
\end{figure}

We will work with a re-parameterization of the open ASEP system by $A,B,C,D$ and $q$, where
\be\label{eq:defining ABCD}
A=\phi_+(\beta,\de),\quad B=\phi_-(\beta,\de),\quad C=\phi_+(\alpha,\gamma),\quad D=\phi_-(\alpha,\gamma), 
\ee 
and
\be \label{eq:phi}
\phi_{\pm}(x,y)=\frac{1}{2x}\lb1-q-x+y\pm\sqrt{(1-q-x+y)^2+4xy}\rb, \quad \mbox{for }\; x>0\mbox{ and }y\geq 0.
\ee 
The quantities $\frac{A}{1+A}$ and $\frac{1}{1+C}$ have nice
physical interpretations as the `effective densities' near the left and right boundaries of the system, see for example a survey   \cite[Section 6.2]{corwin2022some}.

One can check that \eqref{eq:defining ABCD} gives a bijection between \eqref{eq:conditions open ASEP} and 
\be\label{eq:conditions qABCD}
A,C\geq0,\quad -1<B,D\leq 0,\quad 0\leq q<1.
\ee 
We will assume \eqref{eq:conditions open ASEP} and consequently,  \eqref{eq:conditions qABCD} throughout the paper.

It is known since \cite{derrida1993exact} that as the system size $n\rightarrow\infty$, the asymptotic behavior of open ASEP is governed by parameters $A$ and $C$, which exhibits a phase diagram (Figure \ref{fig:phase}) involving three phases:
\begin{enumerate}
    \item [$\bullet$] (maximal current phase) $A<1$, $C<1$,
    \item [$\bullet$] (high density phase) $A>1$, $A>C$,
    \item [$\bullet$] (low density phase) $C>1$, $C>A$.
\end{enumerate} 

There are also two regions on the phase diagram distinguished by \cite{derrida2002exact,derrida2003exact}:
\begin{enumerate}
    \item [$\bullet$] (fan region) $AC<1$,
    \item [$\bullet$] (shock region) $AC>1$.
\end{enumerate}

We denote by $\mu_n:=\mu_n^{(A,B,C,D)}$ the (unique) stationary measure of open ASEP, which is a probability measure on $(\tau_1,\dots,\tau_n)\in\left\{0,1\right\}^n$, where
$\tau_i\in\left\{0,1\right\}$ is the occupation variable on site $i$, for $i=1,\dots,n$. For any $1\leq m\leq n$, the marginal distribution of $\mu_n$ on the sublattice $\{1,\dots,m\}$ is denoted by $\mu_{n|m}:=\mu_{n|m}^{(A,B,C,D)}$, which is a probability measure on $(\tau_1,\dots,\tau_m)\in\{0,1\}^m$. For any two probability measures $\k$ and $\k'$ on $\{0,1\}^m$, we denote their total variation distance by 
$$
\dtv(\k,\k'):=\frac{1}{2}\sum_{x\in\{0,1\}^m}|\k(x)-\k'(x)|=\max_{A\subseteq\{0,1\}^m} |\k(A)-\k'(A)|.
$$
We equip $\{0,1\}^{\ZZ_+}$ with the infinite product $\sigma$-algebra. The measure $\mu_n$ on $\{0,1\}^n$ can be regarded as a probability measure on $\{0,1\}^{\ZZ_+}$ by setting $\tau_i=0$ for $i\geq n+1$, which we also denote by $\mu_n$. As mentioned in the preface, we will  characterize the weak limit $\mu_{\infty}$ of measures $\mu_n$ as $n\rightarrow\infty$. We will also provide a bound of total variation distance between probability measures $\mu_{n|m}$ and the marginal of $\mu_{\infty}$ on the sublattice $\{0,1\}^m$, for  $1\leq m\leq n$. This total variation bound will in particular imply that for  sequences $m_n$, $n=1,2,\dots$ of growth rate $n/\log n$, the total variation distance tends to $0$.

We now state our main theorem in the low density phase:
\begin{theorem}\label{thm:LD}
    In the low density phase $C>A$, $C>1$, as $n\rightarrow\infty$, the open ASEP stationary measure $\mu_n$ weakly converges to the product Bernoulli measure on the semi-infinite lattice $\ZZ_+$ with density $\frac{1}{1+C}$. 
    
    We furthermore assume that $A/C\notin\{q^l:l\in\ZZ_+\}$ if $A\geq1$. Then there exists $H>0$  depending on $A,B,C,D$ and $q$ and $\th\in(0,1)$ depending on $A,C$ and $q$ such that for any $1\leq m\leq n$ we have:
     \be\label{eq:bound LD}
    \dtv\lb\mu_{n|m},\ber_m\lb\frac{1}{1+C}\rb\rb\leq \th^n (Hm)^{3m},
    \ee 
    where we use $\ber_k(\rho)$ to denote the product Bernoulli measure on the lattice $\{1,\dots,k\}$ with density $\rho$, for $ \rho\in[0,1]$ and $k\in\ZZ_+$.    
    
    As a corollary of the bound \eqref{eq:bound LD} above, in view of Lemma \ref{lem:tech},  there exists $s>0$ depending on $A,C$ and $q$ such that for any sequence $\{m_n\}_{n=1}^{\infty}$  satisfying  $m_n\leq s\frac{n}{\log n}$, the total variation distance between $\mu_{n|m_n}$ and $\ber_{m_n}\lb\frac{1}{1+C}\rb$ converges to zero as $n\rightarrow\infty$.
\end{theorem}

 In the high density phase,  as mentioned in the preface, the limiting measure is in general no longer product Bernoulli. Instead, the limiting measure will be characterized in the next definition  in terms of the Askey--Wilson signed measures introduced in \cite{wang2023askey}.

\begin{definition}\label{def:HD limit measure}
    In the high density phase $A>1$, $A>C$, we assume that $C/A\notin\{q^l:l\in\ZZ_+\}$ if $C\geq1$. We define probability measures $\la_m$ on $\{0,1\}^m$ for $m\in\ZZ_+$ by their joint generating functions: For some $\ep>0$ and for any $t_1\leq\dots\leq t_m$ within the interval $(1-\ep,1]$,
\be\label{eq:lambda}
\EE_{\la_m}\le\prod_{i=1}^mt_i^{\tau_i}\re=\frac{A^m}{(1+A)^{2m}}\int_{\RR^m}\prod_{i=1}^m\lb1+t_i+2\sqrt{t_i}x_i\rb
        \pi^{(A,1/A,C,D)}_{t_1,\dots,t_m}\lb\d x_1,\dots,\d x_m\rb,
\ee
where $\pi_{t_1,\dots,t_m}\lb\d x_1,\dots,\d x_m\rb$ on the RHS above is the multi-dimensional Askey--Wilson signed measure (see Section \ref{subsec:Backgrounds} for a brief review).
We will prove in Theorem \ref{thm:HD} below that there exists $\ep>0$ depending on $A,B,C,D$ and $q$ such that for all $m\in\ZZ_+$, the expression on the RHS above is indeed the generating function of a probability measure $\la_m$ on $\{0,1\}^m$, and that the marginal distribution of $\la_{m+1}$ on the leftmost sublattice $\{1,\dots,m\}$ coincides with $\la_m$. We define the probability measure $\lambda$ on $\{0,1\}^{\ZZ_+}$ by requiring that its marginal distribution on the sublattice $\{1,\dots,m\}$ equals $\lambda_m$ for all $m = 1, 2, \dots$.
\end{definition}

 The next result is our main theorem in the high density phase.

\begin{theorem}\label{thm:HD}
In the high density phase $A>1$, $A>C$, we assume that $C/A\notin\{q^l:l\in\ZZ_+\}$ if $C\geq1$.
Then the probability measures $\la_m$ on $\{0,1\}^m$ for $m\in\ZZ_+$ and the probability measure $\la$ on $\{0,1\}^{\ZZ_+}$ in Definition \ref{def:HD limit measure} are well-defined.  
 Furthermore, as $n\rightarrow\infty$, the open ASEP stationary measure $\mu_n$ on $\{0,1\}^n$ weakly converges to the measure $\la$ on $\{0,1\}^{\ZZ_+}$. 

Moreover,  there exists $H>0$  depending on $A,B,C,D$ and $q$ and $\th\in(0,1)$ depending on $A,C$ and $q$ such that for any $1\leq m\leq n$ we have:
   \be\label{eq:bound HD}
    \dtv\lb\mu_{n|m},\la_m\rb\leq \th^n (Hm)^{3m},
   \ee
   
    As a corollary of the bound \eqref{eq:bound HD} above, in view of Lemma \ref{lem:tech},  there exists $s>0$ depending on $A,C$ and $q$ such that for any sequence $\{m_n\}_{n=1}^{\infty}$  satisfying  $m_n\leq s\frac{n}{\log n}$, the total variation distance between $\mu_{n|m_n}$ and $\la_{m_n}$ converges to zero as $n\rightarrow\infty$.
\end{theorem} 

Theorem \ref{thm:LD} and Theorem \ref{thm:HD} above will be proved in Section \ref{subsec:proof of theorems}.

\begin{remark}
    We note that Theorem \ref{thm:LD} and Theorem \ref{thm:HD}, which respectively concern the low density and high density phases, are not related by the particle-hole duality (Lemma \ref{lem:particle hole}) of the open ASEP stationary measure. The particle-hole dual of Theorem \ref{thm:LD} in the low density phase would correspond to a statement in the high density phase, but for the limit as $n\rightarrow\infty$ of a certain transformation of the open ASEP stationary measure---that is, the measure $\widetilde{\mu_n}$ on $\{0,1\}^n$ defined by
    $\widetilde{\mu_n}(\tau_1,\dots,\tau_n)=\mu_n(1-\tau_n,\dots,1-\tau_1)$ for any $\tau_1,\dots,\tau_n\in\left\{0,1\right\}$, where $\mu_n$ is the open ASEP stationary measure. The two limiting measures on $\{0,1\}^{\ZZ_+}$ obtained respectively from $\widetilde{\mu_n}$ and $\mu_n$ on $\{0,1\}^n$ are fundamentally different.
\end{remark}

\begin{remark}
In our bounds \eqref{eq:bound LD} and \eqref{eq:bound HD} of the total variation distance, the constant $\th\in(0,1)$ can be given by
   $$\th=\frac{2+\max\lb2, qC+\lb qC\rb^{-1}, \max \lb A,1\rb+\max\lb A,1\rb^{-1} 
   \rb}{2+ C+C^{-1} }$$
   in the low density phase and by the same formula with $A$ and $C$ swapped in the high density phase. 
The constant $s>0$ appearing in the growth rate $m_n\leq sn/\log n$ can be given by $s=-(\log\th)/3$. These constants can be observed from our proofs of the main theorems in Section \ref{subsec:proof of theorems}.
   
\end{remark}
\begin{remark}
    Our total variation distance bounds \eqref{eq:bound LD} and \eqref{eq:bound HD} are not expected to be optimal. Hence the growth rate  $n/\log n$ of sequences $\{m_n\}_{n=1}^{\infty}$ induced by those bounds are also not optimal. It would be an interesting question to ask for the optimal growth rate. Moreover, in the maximal current phase $A<1$, $C<1$,  we do not know how to use our methods to obtain a total variation distance bound that is not too loose and that yields a growth rate of $\{m_n\}_{n=1}^{\infty}$ that is not too slow. Hence we do not cover this phase in the present paper. It remains an interesting question to find effective total variation distance bounds and the optimal growth rate of $\{m_n\}_{n=1}^{\infty}$ in the maximal current phase.  
\end{remark}

The next result determines exactly when the limiting measure $\lambda$ on $\{0,1\}^{\ZZ_+}$ in the high density phase is a product Bernoulli measure.

\begin{proposition}\label{prop:product Bernoulli} Assume the same conditions as in Definition \ref{def:HD limit measure}, i.e., we are in the high density phase $A>1$, $A>C$, and that $C/A \notin \{q^l : l \in \ZZ_+\}$ if $C \geq 1$. Then the measure $\lambda$ on $\{0,1\}^{\ZZ_+}$ introduced in Definition \ref{def:HD limit measure} is a product Bernoulli measure if and only if $AC = 1$, in which case it has density $A/(1+A)$. \end{proposition}

Proposition \ref{prop:product Bernoulli} will be proved in Section \ref{sec:proof of product Bernoulli}. We note that when  $AC = 1$, it is known that the open ASEP stationary measure $\mu_n$ on $\{0,1\}^n$ is a product Bernoulli measure with density $A/(1+A)$, see for example \cite[Appendix A]{enaud2004large}. Therefore the limiting measure $\lambda$ on $\{0,1\}^{\mathbb{Z}_+}$ is also a product Bernoulli measure with the same density.

\subsection{Comparision with a related result}
As mentioned in the preface, a recent work by Nestoridi--Schmid \cite{nestoridi2023approximating} established a result which is of the same type as our Theorem \ref{thm:LD} in the low density phase. We rephrase their result using our notation as follows.
\begin{theorem}[Theorem 1.4 in \cite{nestoridi2023approximating}]
\label{thm:cite dominik}
Assume that the open ASEP jump rates $\gamma=\delta=0$. Within the low density phase $C>A$, $C>1$ and the shock region $AC>1$, assume furthermore that there exist $\beta'$ and $\beta''$ satisfying $\beta'\leq\beta\leq\beta''$ and some $k\in\ZZ_+$ such that the respective parameters $A':=\phi_+(\beta',\delta)$ and $A'':=\phi_+(\beta'',\delta)$ (recall equations \eqref{eq:defining ABCD} and \eqref{eq:phi} defining $A,B,C,D$) satisfy $C>\max(A',A'',1)$ and $A'Cq^k=A''Cq^{k-1}=1$. Then for any  sequence $\{m_n\}_{n=1}^{\infty}\subset\ZZ_+$ satisfying $n-m_n\rightarrow\infty$ as $n\rightarrow\infty$,  
$$\lim_{n\rightarrow\infty}\dtv\lb \mu_{n|m_n},\ber_{m_n}\lb\frac{1}{1+C}\rb\rb=0.$$ 
\end{theorem}

This theorem was proved in \cite{nestoridi2023approximating} using an entirely different method from the present paper. Their method relies on a characterization of the open ASEP stationary measures as convex linear combinations of specific Bernoulli shock measures, under the special condition $ACq^k=1$, as shown in Jafarpour--Masharian \cite{jafarpour2007matrix}. When comparing it with Theorem \ref{thm:LD} in the low density phase, Theorem \ref{thm:cite dominik} (i.e., \cite[Theorem 1.4]{nestoridi2023approximating}) provides a significantly faster growth rate for the sequence $\{m_n\}_{n=1}^{\infty}$, albeit within a smaller parameter range, assuming $\gamma=\delta=0$, and within a subregion of the shock region.

\subsection{Outline of the paper} 
Section \ref{subsec:Backgrounds} reviews the Askey--Wilson signed measures and open ASEP stationary measures, along with their useful properties.  In Section \ref{subsec:proof of theorems} we prove the main theorems using technical results that are provided in the three appendices.
 In Section \ref{sec:proof of product Bernoulli} we prove Proposition \ref{prop:product Bernoulli}. 
Appendix \ref{sec:bounding TV distance by generating functions} establishes a bound of the total variation distance between two probability measures by their generating functions. Appendix \ref{sec:total variation bounds} provides total variation bounds for certain Askey--Wilson signed measures. Appendix \ref{sec:time reversal} proves a special symmetry of multi-dimensional Askey--Wilson signed measures known as the time reversal.

 \begin{acks}[Acknowledgments]
We thank Wlodek Bryc for bringing the paper \cite{nestoridi2023approximating} to our attention, for pointing out an error in the previous version of this manuscript, and for   helpful discussions, which, in particular, contributed to Appendix \ref{sec:time reversal}. We thank Ivan Corwin, Evita Nestoridi, Dominik Schmid, Yizao Wang  and Jacek Weso\l owski for helpful conversations.
\end{acks}
\begin{funding}
The author was partially supported by Ivan Corwin's NSF grants DMS:1811143, DMS:2246576, Simons Foundation Grant 929852, and the Fernholz Foundation's `Summer Minerva Fellows' program.
\end{funding}

\section{Proofs of the main theorems} 
We review some  background  in Section \ref{subsec:Backgrounds} and prove the main theorems in Section \ref{subsec:proof of theorems}.
 In Section \ref{sec:proof of product Bernoulli} we will prove Proposition \ref{prop:product Bernoulli}. 

\subsection{Background}
\label{subsec:Backgrounds}
We first review the definition and some results of Askey--Wilson signed measures following \cite{wang2023askey}. The Askey--Wilson signed measures were introduced in \cite{wang2023askey} generalizing the Askey--Wilson processes from the  earlier works \cite{bryc2010askey,bryc2017asymmetric}. 
We later review some useful properties of the open ASEP stationary measures.

\begin{definition}[Definition 2.1 and Definition 2.2 in \cite{wang2023askey}]\label{def:AW}
Assume $q\in[0,1)$. We denote by $\Omega$ the  set of parameters $(a,b,c,d)\in\CC^4$ satisfying the following three assumptions:
\begin{enumerate} 
    \item  [\textbf{(1)}] $a,b$ are real, and $c,d$ are either real or form complex conjugate pair; $ab<1$ and $cd<1$,
    \item  [\textbf{(2)}]for any two  distinct $\e,\f\in\{a,b,c,d\}$ such that $|\e|,|\f|\geq 1$,  we have  $\e/\f\notin\{q^l:l\in\mathbb{Z}\}$,
    \item   [\textbf{(3)}]  $q^labcd\neq1$ for all $l\in\NN$,  where $\NN:=\{0,1,\dots\}$.
\end{enumerate}
For $(a,b,c,d)\in\Omega$, the Askey--Wilson signed measure is of mixed type:
    \be\label{eq:AW signed}\nu(\d x;a,b,c,d)=f(x;a,b,c,d)\one_{|x|<1}\d x+\sum_{x\in F(a,b,c,d)}
p(x) \delta_{x},\ee
where the continuous part density is defined as, for $x=\cos\theta\in(-1,1)$,
\be\label{eq:continuous part density}f(x;a,b,c,d)=\frac{(q, ab, ac, ad, bc, bd, cd)_{\infty}}{2\pi
(abcd)_{\infty}\sqrt{1-x^2}} \lv\frac{(e^{2\i\theta
})_{\infty}}{(ae^{\i\theta}, be^{\i\theta}, ce^{\i\theta},
de^{\i\theta})_{\infty}} \rv^2. 
\ee 
Here and below, for complex $z$ and $n\in\NN\cup\{\infty\}$, we use the $q$-Pochhammer symbol:
$$
(z)_n=(z;q)_n=\prod_{j=0}^{n-1}\,(1-z q^j), \quad (z_1,\cdots,z_k)_n=(z_1,\cdots,z_k;q)_n=\prod_{i=1}^k(z_i;q)_n.
$$
The set of atoms $F(a,b,c,d)$ is generated by each $\e\in\{a,b,c,d\}$ with  $|\e|\geq 1$. One can observe that the assumption $(a,b,c,d)\in\Omega$ guarantee that such $\e$ is a real number. When $\e=a$ the
corresponding atoms are 
$$
y_{k}^\aa=y_{k}^\aa(a,b,c,d)=\frac{1}{2}\lb a q^k+(aq^k)^{-1}\rb,
$$
with $k\geq0$ such that $|aq^k|\geq1$, and the corresponding
masses are 
\begin{align}\label{eq: p_0}
    p(y_{0}^\aa)&=p_0^\aa(a,b,c,d)=\frac{(a^{-2},bc,bd,cd)_\infty}{(b/a,c/a,d/a,abcd)_\infty},\\
    \label{eq: p_j}
p(y_{k}^\aa)&=p_k^\aa(a,b,c,d) 
 = \frac{p_0^\aa(a,b,c,d)q^k(1-a^2q^{2k})(a^2,ab,ac,ad)_k}{(q)_k(1-a^2)a^k\prod_{l=1}^k\lb(b-q^la)(c-q^la)(d-q^la)\rb},\quad k\geq 1. 
\end{align}
The bold symbol $\aa$ in the superscripts signal that the atom (if exists) is generated by the parameter coming from the $\aa$ (first) position of the four parameters. For $\e\in\{b,c,d\}$, atoms $y_k^\bb$, $y_k^\cc$, $y_k^\dd$ and masses $p(y_k^\bb)$, $p(y_k^\cc)$, $p(y_k^\dd)$ are given by similar formulas with $a$ and $\e$ swapped.  
\end{definition}
\begin{remark}
    Parameter $q\in[0,1)$ will be fixed throughout the paper.
\end{remark}
 
\begin{remark}
    We recall from \cite{wang2023askey} that the Askey--Wilson signed measures are finite signed measures with compact supports in $\RR$. Also they have total mass $1$, in the sense that:
$$\int_{\RR}\nu(\d x;a,b,c,d)=1.$$
\end{remark}

For certain parameters $A,B,C,D\in\RR$ and on some suitable  
`time interval' $I\subset\RR$ (which will be defined later), we will study the following Askey--Wilson signed measures: For $t\in I$,
\be\label{eq:marginal AW signed}\pi^{(A,B,C,D)}_t(\d y):=\nu\lb \d y;A\sqrt{t},B\sqrt{t},\frac{C}{\sqrt{t}},\frac{D}{\sqrt{t}}\rb.\ee 
Define $U_t\subset\RR$ to be the (compact) support of $\pi^{(A,B,C,D)}_t(\d y)$. We will also study the following Askey--Wilson signed measures: For $s,t\in I$, $s<t$ and $x\in U_s$,
\be\label{eq:transition AW signed}P_{s,t}^{(A,B)}(x,\d y):=\nu\lb \d y;A\sqrt{t},B\sqrt{t},\sqrt{\frac{s}{t}}\lb x+\sqrt{x^2-1}\rb,\sqrt{\frac{s}{t}}\lb x-\sqrt{x^2-1}\rb\rb.\ee 
When $s=t\in I$ and $x\in U_s$ we define $P^{(A,B)}_{s,s}(x,\d y)=\delta_x(\d y)$ for convenience.

For any $t_1\leq\dots\leq t_n$ in $I$, we define the `multi-dimensional' Askey--Wilson signed measure:  
    \be\label{eq:definition of multi-time}\pi^{(A,B,C,D)}_{t_1,\dots,t_n}(\d x_1,\dots,\d x_n):=\pi^{(A,B,C,D)}_{t_1}(\d x_1)P^{(A,B)}_{t_1,t_2}(x_1,\d x_2)\dots P^{(A,B)}_{t_{n-1},t_n}(x_{n -1},\d x_n),\ee
which is a finite signed measure with total mass $1$, supported on compact subset $U_{t_1}\times\dots\times U_{t_n}\subset\RR^n$.
\begin{remark}
  We note that in this paper, the term “Askey–Wilson signed measure” is used to refer to several related objects: the signed measure $\nu(\d x; a, b, c, d)$ defined by \eqref{eq:AW signed}; its parameter specializations $\pi_t^{(A,B,C,D)}(\d x)$ and $P_{s,t}^{(A,B)}(x, \d y)$ defined  by \eqref{eq:marginal AW signed} and \eqref{eq:transition AW signed} respectively; and the signed measure $\pi^{(A,B,C,D)}_{t_1,\dots,t_n}(\d x_1,\dots,\d x_n)$ defined by \eqref{eq:definition of multi-time}, which we sometimes also refer to as the multi-dimensional Askey–Wilson signed measure.  

    In the rest of this paper, we will omit the superscripts $(A,B,C,D)$ and $(A,B)$ on the Askey--Wilson signed measures $\pi_t$, $\pi_{t_1,\dots,t_n}$ and $P_{s,t}$ if  no confusion will arise.
\end{remark}

The following result characterizes the open ASEP stationary measures by Askey--Wilson signed measures. This characterization was originally due to \cite{bryc2017asymmetric} in the form of Askey--Wilson processes (see also anterior earlier work \cite{uchiyama2004asymmetric}) and was later generalized in \cite{wang2023askey}.
\begin{theorem} 
\label{thm:characterization}
    Consider the open ASEP stationary measure $\mu_n=\mu_n^{(A,B,C,D)}$ on the lattice $\{1,\dots,n\}$.
    We assume that  $q^lABCD\neq1$ for all $l\in\NN$  and that $A/C\notin\{q^l:l\in\ZZ\}$ if $A,C\geq 1$.
    Then there exists $I=[1,1+\ep)$ for some $\ep=\ep(A,B,C,q)>0$ depending on $A,B,C$ and $q$ such that for any  $1\leq m< n$  and $t_1\leq\dots\leq t_m$ in $I$, we have:
    \begin{multline*}
        \EE_{\mu_n}\le\prod_{i=1}^mt_{i}^{\tau_{n-m+i}}\re\\=\frac{\int_{\RR^{m+1}}(2+2x)^{n-m}\prod_{i=1}^m(1+t_i+2\sqrt{t_i}x_i)\pi_{1,t_1,\dots,t_m}\lb\d x, \d x_1,\dots,\d x_m\rb}{\int_{\RR}(2+2x)^n\pi_1(\d x)}.
    \end{multline*} 
\end{theorem}

 We note that the formula above is the generating function for the open ASEP stationary measure $\mu_n$ on $\{0,1\}^n$, which involves the multi-dimensional Askey–Wilson signed measure parameterized by $(A, B, C, D)$. In contrast, formula \eqref{eq:lambda} in Definition \ref{def:HD limit measure} gives the generating function for the limiting measure $\lambda_m$ on $\{0,1\}^m$, which involves the multi-dimensional Askey–Wilson signed measure parameterized by $(A, 1/A, C, D)$.   
\begin{proof}
    This theorem follows from \cite[Theorem 1.1]{wang2023askey} except for the fact that $\ep>0$ can be taken as  depending on $A,B,C$ and $q$ but not on $D$. By the arguments in \cite{wang2023askey}, we only need to show that there exists $\ep=\ep(A,B,C,q)>0$ such that for any  $s<t$ in $I=[1,1+\ep)$ and $x\in U_s$, we have
    $(a,b,c,d)\in\Omega$, where 
    \be \label{eq:abcd1}
    a=A\sqrt{t},\quad b=B\sqrt{t},\quad c=\sqrt{\frac{s}{t}}\lb x+\sqrt{x^2-1}\rb,\quad d=\sqrt{\frac{s}{t}}\lb x-\sqrt{x^2-1}\rb 
    \ee so that $P_{s,t}^{(A,B)}(x,\d y)=\nu\lb \d y;a,b,c,d\rb$.
    Assumption \textbf{(1)} and \textbf{(3)} in Definition \ref{def:AW} always hold since $a$ and $b$ are real numbers; $c$ and $d$ form a complex conjugate pair; $ab=ABt\leq 0$; $cd=s/t<1$ and $abcd=ABs\leq 0$. 

    Recall that we have assumed $A/C\notin\{q^l:l\in\ZZ\}$ if $A,C\geq 1$. One can choose $\ep=\ep(A,B,C,q)>0$ such that $1+\ep<\min\lb 1/q,1/B^2\rb$ and $1+\ep<1/A^2$ if $A<1$, and that the interval $I=[1,1+\ep)$ does not contain elements in $\{Cq^l/A:l\in\ZZ\}$ if $A,C\geq 1$. In the following we will show that for any $s<t$ in $I$ and $x\in U_s$, assumption \textbf{(2)} in Definition \ref{def:AW} holds for $(a,b,c,d)$ given by \eqref{eq:abcd1}. Note that we always have $b=B\sqrt{t}\in(-1,0]$ for $t\in I$. Also, since $B\sqrt{s}, D/\sqrt{s}\in(-1,0]$, any possible atom in $U_s$ is generated by either $A\sqrt{s}$ or $C/\sqrt{s}$. We split the proof into three cases depending on $x\in U_s$:
\begin{enumerate}
    \item [Case 1.] Let $x\in[-1,1]$ then $c$ and $d$ are complex conjugate pairs with norm $<1$, where the norm of a complex number $z$ refers to its absolute value $|z|$.  Hence at most one element in $\{a,b,c,d\}$ has norm $\geq1$ and assumption \textbf{(2)} vacuously holds.
    \item [Case 2.] Let $x=\frac{1}{2}\lb q^jA\sqrt{s}+\lb q^jA\sqrt{s}\rb^{-1}\rb$ for $j\in\NN$ and $q^jA\sqrt{s}>1$. Then $c=q^jAs/\sqrt{t}$ and $d=1/\lb q^jA\sqrt{t}\rb<1$. Only $a$ and $c$ in $\{a,b,c,d\}$ can have norm $\geq1$. We have $c/a=q^js/t\notin\{q^l:l\in\ZZ\}$ since $s/t\in(q,1)$. Hence assumption \textbf{(2)} holds.
    \item [Case 3.] Let $x=\frac{1}{2}\lb q^jC/\sqrt{s}+\lb q^jC/\sqrt{s}\rb^{-1}\rb$ for $j\in\NN$ and $q^jC/\sqrt{s}>1$. Then $c=q^jC/\sqrt{t}$ and $d=s/(q^jC\sqrt{t})<1$. Only $a$ and $c$ in $\{a,b,c,d\}$ can have norm $\geq1$. If $A<1$ then $a=A\sqrt{t}<1$ and assumption \textbf{(2)} vacuously holds. If $C<1$ then $c=q^jC/\sqrt{t}<1$ and assumption \textbf{(2)} vacuously holds. If $A,C\geq 1$ then $c/a=q^jC/(At)\notin\{q^l:l\in\ZZ\}$ by our assumption, hence assumption \textbf{(2)} holds. 
\end{enumerate} 

We conclude the proof.
\end{proof}

The next result is a basic property of the Askey--Wilson signed measure $P_{s,t}(x,\d y)$:
\begin{lemma}[Lemma 2.14 and Lemma 2.15 in \cite{wang2023askey}]\label{lem:cite}
    Assume  that $q^lABCD\neq1$ for all $l\in\NN$ and that $A/C\notin\{q^l:l\in\ZZ\}$ if $A,C\geq 1$. Choose the interval $I=[1,1+\ep)$ from Theorem \ref{thm:characterization}.
    Then for any $s\leq t$ in $I$ and any $x\in U_s$, the Askey--Wilson signed measure $P_{s,t}(x,\d y)$ is supported on $U_t$. In the high density phase, we have $P_{s,t}\lb y_0(s),\d y\rb=\delta_{y_0(t)}(\d y)$, where we denote by $y_0(t)$ the largest atom in $U_t$ for any $t\in I$.
\end{lemma}

We next recall some useful properties of the open ASEP stationary measure.
The following result is well-known as the particle-hole duality:
\begin{lemma}[Particle-hole duality]\label{lem:particle hole}
    The stationary measures $\mu_n^{(A,B,C,D)}$ and $\mu_n^{(C,D,A,B)}$ for open ASEP on $\{1,\dots,n\}$ are related by a transformation of occupation variables $$\widehat{\tau}_i:=1-\tau_{n+1-i}\quad\text{ for }i=1,\dots,n$$
\end{lemma}

See for example \cite[Section 4.2]{wang2023askey} for the proof of the above result.

The open ASEP stationary measure is known to be continuous with respect to its parameters:
\begin{lemma}\label{lem:continuity}
    As a probability measure on $\{0,1\}^n$, the open ASEP stationary measure $\mu_n^{(A,B,C,D)}$ depends continuously on its parameters $A,B,C,D$ and $q$.
\end{lemma}
\begin{proof}
    It is shown in \cite[Remark 1.9]{barraquand2023stationary} that the open ASEP stationary measure depends real analytically on the jump rates $\alpha,\beta,\gamma,\delta$ and $q$ in the region of the parameter space where the stationary measure is unique. The continuity with respect to parameters $A,B,C,D$ then follows from \cite[equation (2.4)]{bryc2017asymmetric}:
    \begin{multline*}
        \alpha=\frac{1-q}{(1+C)(1+D)},\quad \beta=\frac{1-q}{(1+A)(1+B)},\quad \\\gamma=\frac{-(1-q)CD}{(1+C)(1+D)},\quad \delta=\frac{-(1-q)AB}{(1+A)(1+B)}.\end{multline*}
\end{proof}

The following result is known as the `stochastic sandwiching' of open ASEP stationary measure, which is first introduced by \cite{corwin2021stationary}:
\begin{lemma}[Lemma 4.1 in \cite{corwin2021stationary}]\label{lem:sandwiching}
    Fix $q\in[0,1)$ and consider real numbers
    $$0<\alpha' \leq\alpha'',\quad\beta' \geq\beta''>0,\quad\gamma' \geq\gamma''\geq0,\quad0\leq\delta' \leq\delta''.$$

Consider the open ASEP on the lattice $\{1,\dots,n\}$ with rates $(q,\alpha',\beta',\gamma',\delta')$ and $(q,\alpha'',\beta'',\gamma'',\delta'')$. We denote the stationary measures as $\mu_n'$ and $\mu_n''$. The corresponding occupation variables are denoted as $(\tau'_1,\dots,\tau'_n)$ and  $(\tau''_1,\dots,\tau''_n)$. Then there exists a coupling of $\mu'_n$ and $\mu_n''$ such that almost surely $\tau'_i\leq\tau_i''$ for $i=1,\dots,n$.

As a corollary, for any $t_1,\dots,t_n\geq1$, we have
$$\EE_{\mu'_n}\le\prod_{i=1}^nt_i^{\tau_i}\re\leq\EE_{\mu''_n}\le\prod_{i=1}^nt_i^{\tau_i}\re.$$
\end{lemma}

\subsection{Proof of Theorem \ref{thm:LD} and Theorem \ref{thm:HD}}
\label{subsec:proof of theorems}
In this subsection we will demonstrate the proofs of our main results: Theorem \ref{thm:LD} in the low density phase and Theorem \ref{thm:HD} in the high density phase. 

For reasons that will become clear later, we first prove the results modulo a particle-hole duality. The main parts of these theorems (modulo the particle-hole duality) are stated in  Theorem \ref{thm:modulo particle hole} below, which combines both the high and low density phases.

\begin{definition}\label{def:measure etam}
    Assume $\max(A,C)>1$ and that $A/C\notin\{q^l:l\in\ZZ\}$ if $A,C\geq 1$. We define probability measures $\eta_m$ on $\{0,1\}^m$ for $m\in\ZZ_+$ by their joint generating functions:  For some $\ep>0$ and for any $t_1\leq\dots \leq t_m$ within the interval $[1,1+\ep)$,
    \begin{multline}
    \label{eq:def of eta}
\EE_{\eta_m}\le\prod_{i=1}^mt_i^{\tau_i}\re\\=\int_{\RR^m}\prod_{i=1}^m\frac{1+t_i+2\sqrt{t_i}x_i}{2+2y_0(1)}
        P_{1,t_1}\lb y_0(1),\d x_1\rb P_{t_1,t_2}\lb x_1,\d x_2\rb \dots P_{t_{m-1},t_m}\lb x_{m-1},\d x_m\rb,
\end{multline} 
 where we recall from Lemma \ref{lem:cite} that $y_0(1)$ is the largest atom in the (compact) support $U_1$ of the Askey--Wilson signed measure $\pi_1(\d y)=\nu(\d y;A,B,C,D)$, and $P_{s,t}(x,\d y)$ was defined by \eqref{eq:transition AW signed}.
 We will prove in Theorem \ref{thm:modulo particle hole}  that there exists $\ep>0$ depending on $A,B,C,D$ and $q$ such that for all $m\in\ZZ_+$, the probability measure $\eta_m$ on $\{0,1\}^m$ are well-defined.
\end{definition}

\begin{theorem}\label{thm:modulo particle hole}
    Assume $\max(A,C)>1$ and that $A/C\notin\{q^l:l\in\ZZ\}$ if $A,C\geq 1$.   Then the probability measures $\eta_m$ on $\{0,1\}^m$ for $m\in\ZZ_+$ in Definition \ref{def:measure etam} are well-defined, and  the marginal distribution of $\eta_{m+1}$ on the leftmost sublattice $\{1,\dots,m\}$ coincides with $\eta_m$.   

Denote the marginal distribution of open ASEP stationary measure $\mu_n=\mu_n^{(A,B,C,D)}$ on the last $m$ sites $\{n-m+1,\dots,n\}$ by $\mu_{n,m}$. Then there exists $H>0$ depending on $A,B,C,D$ and $q$ and $\th\in(0,1)$ depending on $A,C$ and $q$  such that for all  $1\leq m\leq n$,
\be  \label{eq:TV bound final}
\dtv\lb\mu_{n,m},\eta_m\rb\leq\th^n (Hm)^{3m}. 
\ee  
\end{theorem}
\begin{remark}\label{rmk:mu}
    We note the difference between notations $\mu_{n|m}$ and $\mu_{n,m}$. The first notation denotes the marginal  of $\mu_n$ on the first $m$ sites and the second one denotes the marginal on the last $m$ sites.
\end{remark}

The proof of the above theorem will constitute the major technical component of this section. Before commencing with the proof, we will state two results which will be needed in the proof. 

The following result provides a delicate bound of the total variation of Askey--Wilson signed measures $P_{s,t}(x,\d y)$: 

\begin{proposition}\label{prop:TV bound of AW signed measures1}
    Assume $A,C\geq0$, $-1<B,D\leq 0$ and $q\in[0,1)$.
    Assume also  that $q^lABCD\neq1$ for all $l\in\NN$ and that $A/C\notin\{q^l:l\in\ZZ\}$ if $A,C\geq 1$. Then there exist positive constants $K\geq1$ and $\ep$ depending on $A,B,C$ and $q$ but not on $D$, such that for any $s<t$ in $I=[1,1+\ep)$ and $x\in U_s$, the total variation of the Askey--Wilson signed measure $P_{s,t}(x,\d y)$ is bounded from above by $\frac{K}{(t-s)^2}$. 
\end{proposition}
\begin{remark}\label{rmk:tv bound}
The above result will appear again as Proposition \ref{prop:TV bound of AW signed measures}, which will be proved in Appendix \ref{sec:total variation bounds}. It is worth noting that a strictly weaker version of this total variation bound has previously appeared in \cite{wang2023askey}. In particular, the total variation bound of $P_{s,t}(x,\mathrm{d}y)$ provided by \cite[Proposition A.1]{wang2023askey} can be derived as a simple corollary of the above result, but the converse does not hold. The derivation of the bound here requires a delicate analysis of atom masses, see Appendix \ref{sec:total variation bounds}.\end{remark}

To bound the total variation distance between two probability measures (in our case, $\mu_{n,m}$ and $\eta_m$) on $\{0,1\}^m$, we will use the following bound of this total variation distance by generating functions:

\begin{proposition}\label{prop:bounding distance two prob measures1}
    Let $\k$ and $\k'$ be probability measures on $\{0,1\}^m$. Then for any set of numbers $0<t_{i,0}<t_{i,1}$ for $i=1,\dots,m$, we have:
    $$ 
    \dtv(\k,\k')
    \leq \frac{1}{2}\prod_{i=1}^m\frac{1+t_{i,1}}{t_{i,1}-t_{i,0}}\sum_{\up_1,\dots,\up_m\in\{0,1\}}
    \lv\EE_{\k}\le\prod_{i=1}^mt_{i,\up_i}^{\tau_i}\re - \EE_{\k'}\le\prod_{i=1}^mt_{i,\up_i}^{\tau_i}\re\rv, $$
    where $\tau_i\in\{0,1\}$ is the occupation variable on the site $i$, for $i=1,\dots,m$.
\end{proposition}
\begin{remark}
The above result will appear again as Proposition \ref{prop:bounding distance two prob measures} which will be proved in Appendix \ref{sec:bounding TV distance by generating functions}. As one can observe, this constitutes a general bound applicable to any two probability measures on $\{0,1\}^m$. While it is possible that this result is already known, we have been unable to locate it in references.\end{remark}

We now begin the proof of Theorem \ref{thm:modulo particle hole}:
\begin{proof}[Proof of Theorem \ref{thm:modulo particle hole}]
    We divide the proof into three steps.\\

{\bf \raggedleft \underline{Step 1.}} 
In this step, under an extra assumption that $q^lABCD\neq1$ for all $l\in\NN$, we prove that there exists  $\ep>0$ such that, for $1\leq t_1\leq\dots \leq t_m< 1+\ep$, we have:
\begin{multline}\label{eq:only need to show step 1}
    \lim_{n\rightarrow\infty}\EE_{\mu_{m,n}}\le\prod_{i=1}^mt_i^{\tau_i}\re\\=\int_{\RR^m}\prod_{i=1}^m\frac{1+t_i+2\sqrt{t_i}x_i}{2+2y_0(1)}
        P_{1,t_1}\lb y_0(1),\d x_1\rb P_{t_1,t_2}\lb x_1,\d x_2\rb \dots P_{t_{m-1},t_m}\lb x_{m-1},\d x_m\rb.
\end{multline} 
The above equation in particular implies that its RHS (i.e., the RHS of \eqref{eq:def of eta}) indeed defines a probability measure $\eta_m$ on $\{0,1\}^m$ that is the limit of measures $\mu_{m,n}$ as $n\rightarrow\infty$. As a corollary, the marginal distribution of $\eta_{m+1}$ on the first $m$ sites coincides with $\eta_m$. 

We choose $\ep>0$ and $I=[1,1+\ep)$ according to Theorem \ref{thm:characterization} and Proposition \ref{prop:TV bound of AW signed measures1}.  By Theorem \ref{thm:characterization}, 
we have that for any $1\leq m< n$ and $1=t_0\leq t_1\leq\dots\leq t_m<1+\ep$,
\be\label{eq:generating as G1'G2'}\EE_{\mu_{n,m}}\le\prod_{i=1}^mt_i^{\tau_i}\re=\EE_{\mu_n}\le\prod_{i=1}^mt_{i}^{\tau_{n-m+i}}\re=\frac{\G_1'}{\G_2'},\ee
where we define:
\begin{align}\label{eq:G1'G2'}
    \begin{split}
        \G_1':=&\frac{1}{\lb 2+2y_0(1)\rb^n}\int_{\RR^{m+1}}(2+2x)^{n-m}\prod_{i=1}^m(1+t_i+2\sqrt{t_i}x_i)\pi_{1,t_1,\dots,t_m}\lb\d x, \d x_1,\dots,\d x_m\rb,\\
        \G_2':=&\frac{1}{\lb 2+2y_0(1)\rb^n}\int_{\RR}(2+2x)^n\pi_1(\d x).
    \end{split}
\end{align}
We write the RHS of \eqref{eq:only need to show step 1} as $\G_1/\G_2$, where  
\begin{align}\label{eq:G1G2}
    \begin{split}
        \G_1:=&\int_{\RR^m}\prod_{i=1}^m\frac{1+t_i+2\sqrt{t_i}x_i}{2+2y_0(1)}\pi_{1,t_1,\dots,t_m}\lb\{y_0(1)\}, \d x_1,\dots,\d x_m\rb \\
        =&\frac{1}{\lb 2+2y_0(1)\rb^n} \\
        & \times \int_{\RR^{m}}(2+2y_0(1))^{n-m}\prod_{i=1}^m(1+t_i+2\sqrt{t_i}x_i)\pi_{1,t_1,\dots,t_m}\lb\{y_0(1)\}, \d x_1,\dots,\d x_m\rb,\\
        \G_2:=&\pi_1\lb\{y_0(1)\}\rb = \frac{1}{\lb 2+2y_0(1)\rb^n}\int_{\{y_0(1)\}}(2+2x)^n\pi_1(\d x).
    \end{split}
\end{align}
We need to bound the difference between $\G_1/\G_2$ and $\G_1'/\G_2'$. Under $|\G_2|>|\G_2'-\G_2|$,  we have:
\be \label{eq:gamma bound}
\begin{split} \left\lvert\frac{\G_1'}{\G_2'}-\frac{\G_1}{\G_2}\right\rvert 
    &=\frac{|\G_1'\G_2-\G_1\G_2'|}{|\G_2'\G_2|} 
     \leq \frac{|\G_1'\G_2-\G_1\G_2|+|\G_1\G_2-\G_1\G_2'|}{|\G_2'\G_2|}\\
    &=\frac{|\G_1'-\G_1|}{|\G_2'|}+\frac{|\G_1|}{|\G_2'\G_2|}|\G_2'-\G_2| 
    \leq\frac{|\G_1'-\G_1|}{|\G_2|-|\G_2'-\G_2|}+\frac{|\G_1|}{|\G_2|}\frac{|\G_2'-\G_2|}{|\G_2|-|\G_2'-\G_2|}. 
\end{split}
\ee  
Hence we need to bound three quantities $|\G_1/\G_2|$, $|\G_1'-\G_1|$ and $|\G_2'-\G_2|$. We first recall that $\G_1/\G_2$ is the RHS of \eqref{eq:only need to show step 1}, i.e.,
\be\label{eq:vwev}
\frac{\G_1}{\G_2} =\int_{\RR^m}\prod_{i=1}^m\frac{1+t_i+2\sqrt{t_i}x_i}{2+2y_0(1)}
       P_{1,t_1}\lb y_0(1),\d x_1\rb P_{t_1,t_2}\lb x_1,\d x_2\rb \dots P_{t_{m-1},t_m}\lb x_{m-1},\d x_m\rb.
\ee
In view of Lemma \ref{lem:cite}, the signed measure $$P_{1,t_1}\lb y_0(1),\d x_1\rb\dots P_{t_{m-1},t_m}\lb x_{m-1},\d x_m\rb$$ on the RHS
is supported on $U_{t_1}\times\dots\times U_{t_m}$. 
We next estimate its total variation. We denote the total variation of any signed measure $\nu$ by $\lnor\nu\rnor$.
Notice that for $1\leq i\leq m$ and $x\in U_{t_{i-1}}$, we have $\lnor P_{t_{i-1},t_i}\lb x,\d y\rb\rnor=1$ when $t_{i-1}=t_i$ and, by Proposition \ref{prop:TV bound of AW signed measures1}, $\lnor P_{t_{i-1},t_i}\lb x,\d y\rb\rnor\leq\frac{K}{(t_i-t_{i-1})^2}$ when $t_{i-1}<t_i$. Therefore, by multiplying those bounds together, we have:
\be \label{eq:qwer}\lnor P_{1,t_1}\lb y_0(1),\d x_1\rb P_{t_1,t_2}\lb x_1,\d x_2\rb \dots P_{t_{m-1},t_m}\lb x_{m-1},\d x_m\rb\rnor\leq \prod_{1\leq i\leq m, t_{i-1}<t_i}\frac{K}{(t_i-t_{i-1})^2}.\ee
Notice that $1+t+2\sqrt{t}x>0$ for any $x\in U_t$.  
Define $$r:= \sup_{t\in I} \lb 1+t+2\sqrt{t}y_0(t)\rb$$ 
In view of \eqref{eq:vwev} and \eqref{eq:qwer},  and noting that $2+2y_0(t)\geq4>1$, we have: 
 \be\label{eq:bound on G1 over G2} 
\lv\frac{\G_1}{\G_2}\rv\leq \lb \frac{r}{2+2y_0(1)}\rb^m\prod_{1\leq i\leq m, t_{i-1}<t_i}\frac{K}{(t_i-t_{i-1})^2}\leq r^m\prod_{1\leq i\leq m, t_{i-1}<t_i}\frac{K}{(t_i-t_{i-1})^2}.
\ee 
We next bound the differences $|\G_1'-\G_1|$ and $|\G_2'-\G_2|$.  
By \eqref{eq:G1'G2'} and \eqref{eq:G1G2} we have: 
\begin{multline}\label{eq:step 11}
    \G_1'-\G_1=\frac{1}{\lb 2+2y_0(1)\rb^n}\int_{\lb U_1\setminus\{y_0(1)\}\rb\times U_{t_1}\times\dots\times U_{t_m}}(2+2x)^{n-m}\prod_{i=1}^m(1+t_i+2\sqrt{t_i}x_i)\\\times\pi_{1,t_1,\dots,t_m}\lb\d x, \d x_1,\dots,\d x_m\rb.
\end{multline}
Using Proposition \ref{prop:TV bound of AW signed measures1} in a similar way as above, we have:
\begin{multline*}
    \lnor\pi_{1,t_1,\dots,t_m}\rnor \leq 
        \lnor\pi_1\rnor\prod_{i=1}^m\sup_{x\in U_{t_{i-1}}}\lnor P_{t_{i-1},t_i}\lb x,\d y\rb\rnor \\
         \leq \lnor\pi_1\rnor\prod_{1\leq i\leq m, t_{i-1}<t_i}\frac{K}{(t_i-t_{i-1})^2}.
\end{multline*}
For $t\in I$, we denote by $y_0^*(t)$ the largest element of $U_t\setminus\{y_0(t)\}$. In the support $$\lb U_1\setminus\{y_0(1)\}\rb\times U_{t_1}\times\dots\times U_{t_m}$$ on the RHS integration of \eqref{eq:step 11}, we have $x\leq y^*_0(1)$ and $x_i\leq y_0(t_i)$ for $1\leq i\leq m$. Hence:
\begin{multline*}
0<\frac{(2+2x)^{n-m}\prod_{i=1}^m(1+t_i+2\sqrt{t_i}x_i)}{\lb 2+2y_0(1)\rb^n}\\\leq\frac{(2+2y^*_0(1))^{n-m}r^m}{\lb 2+2y_0(1)\rb^n}< \frac{(2+2y^*_0(1))^nr^m}{\lb 2+2y_0(1)\rb^n}=\th^nr^m,
\end{multline*}
where we denote $\th:=\frac{1+y^*_0(1)}{1+y_0(1)}\in(0,1)$.
Therefore, in view of \eqref{eq:step 11}, we have
\be \label{eq:bound of quotience1} 
|\G_1'-\G_1| \leq \th^nr^m\lnor\pi_1\rnor\prod_{1\leq i\leq m, t_{i-1}<t_i}\frac{K}{(t_i-t_{i-1})^2}. 
\ee 
By \eqref{eq:G1'G2'} and \eqref{eq:G1G2} we also have:
$$\G_2'-\G_2=\frac{1}{\lb 2+2y_0(1)\rb^n}\int_{U_1\setminus\{y_0(1)\}}(2+2x)^n\pi_1(\d x).$$
Using a similar but simpler argument to the one above, we conclude
\[ 
|\G_2'-\G_2| \leq\th^n\lnor\pi_1\rnor.
\] 

We recall that $\G_2=\pi_1\lb\{y_0(1)\}\rb$.
There exists $N=N(A,B,C,D,q)\in\ZZ_+$ such that for any $n\geq N$, we have 
\be \label{eq:bound of quotience2}|\G_2|=\lv\pi_1\lb\{y_0(1)\}\rb\rv>\th^n\lnor \pi_1\rnor\geq|\G_2'-\G_2|.\ee  
Combining the estimates \eqref{eq:gamma bound}, \eqref{eq:bound on G1 over G2}, \eqref{eq:bound of quotience1} and \eqref{eq:bound of quotience2},  we get that for any $n\geq N$, $1\leq m< n$ and $1=t_0\leq t_1\leq\dots\leq t_m<1+\ep$, 
\be\label{eq:estimate to norm in proof}
\left\lvert \frac{\G'_1}{\G'_2}-\frac{\G_1}{\G_2} \right\rvert
   \leq\frac{2\th^n r^m \lnor\pi_1\rnor }{\lv\pi_1\lb\{y_0(1)\}\rb\rv-\th^n\lnor \pi_1\rnor}\prod_{1\leq i\leq m, t_{i-1}<t_i}\frac{K}{(t_i-t_{i-1})^2}.\ee          
For fixed $m$, as $n\rightarrow\infty$, the RHS above converges to $0$. In view of \eqref{eq:generating as G1'G2'} and the definitions of $\G_1$ and $\G_2$ (so that the RHS of \eqref{eq:only need to show step 1} equals $\G_1/\G_2$), we conclude the proof of the convergence \eqref{eq:only need to show step 1}. 
Step 1 is complete.\\

{\bf \raggedleft \underline{Step 2.}} 
In this step, under the extra assumption  that $q^lABCD\neq1$ for all $l\in\NN$, we prove \eqref{eq:TV bound final}: there exists  $H>0$ depending on $A,B,C,D$ and $q$  such that for all  $1\leq m\leq n$,
\be\label{eq:only need step 2}
\dtv\lb\mu_{n,m},\eta_m\rb\leq\th^n (Hm)^{3m}, 
\ee
where we recall $\th =\frac{1+y^*_0(1)}{1+y_0(1)}\in(0,1)$.

To bound the total variation distance of two probability measures on $\{0,1\}^m$, by Proposition \ref{prop:bounding distance two prob measures1}, we only need to bound the difference between their generating functions. 
In view of \eqref{eq:estimate to norm in proof},  we have:
\be \label{eq:bound with generating functions}
    \left\lvert \EE_{\mu_{n,m}}\le\prod_{i=1}^mt_i^{\tau_i}\re-\EE_{\eta_m}\le\prod_{i=1}^mt_i^{\tau_i}\re\right\rvert 
   \leq\frac{2\th^n r^m \lnor\pi_1\rnor }{\lv\pi_1\lb\{y_0(1)\}\rb\rv-\th^n\lnor \pi_1\rnor}\prod_{1\leq i\leq m, t_{i-1}<t_i}\frac{K}{(t_i-t_{i-1})^2}, 
\ee
for any  $n\geq N$, $1\leq m<n$ and $1=t_0\leq t_1\leq\dots\leq t_m<1+\ep$.  
We set $$t_{i,\up,m}:=1+\frac{(i+\up)\ep}{2m}\quad\text{for }  i=1,\dots,m  \text{ and } \up\in\{0,1\}.$$  For any $\up_1,\dots,\up_m\in\{0,1\}$, we take $t_i=t_{i,\up_i,m}$ for $i=1,\dots,m$ in \eqref{eq:bound with generating functions}. In view of the fact that $$t_i-t_{i-1}\in\left\{0,\frac{\ep}{2m},\frac{\ep}{m}\right\},\quad   i=1,\dots,m,$$ we have:
\be \label{eq:generating function bound in proof 2}
\left\lvert \EE_{\mu_{n,m}}\le\prod_{i=1}^mt_{i,\up_i,m}^{\tau_i}\re-\EE_{\eta_m}\le\prod_{i=1}^mt_{i,\up_i,m}^{\tau_i}\re\right\rvert
\leq     
\frac{2\th^n r^m\lnor\pi_1\rnor}{ \lv\pi_1\lb\{y_0(1)\}\rb\rv-\th^n\lnor \pi_1\rnor}
\frac{K^m}{(\ep/(2m))^{2m}}.
\ee 
Using Proposition \ref{prop:bounding distance two prob measures1}, in view of 
$$1+t_{i,1,m}<2+\ep,\quad t_{i,1,m}-t_{i,0,m}=\frac{\ep}{2m},\quad   i=1,\dots,m,$$
  we have:
\be  \label{eq:bound TV distance in proof continuity}
\dtv\lb\mu_{n,m},\eta_m\rb\leq \frac{ \th^n r^m\lnor\pi_1\rnor}{\lv\pi_1\lb\{y_0(1)\}\rb\rv-\th^n\lnor \pi_1\rnor} \frac{K^m}{(\ep/(2m))^{2m}} \lb\frac{2+\ep}{\ep/(2m)}\rb^m2^m.
\ee 
One can choose $H=H(A,B,C,D,q)>0$ sufficiently large so that the RHS of \eqref{eq:bound TV distance in proof continuity} can be bounded above by $\th^n (Hm)^{3m}$ for any $n\geq N$ and  $1\leq m< n$. 
For the cases when $n < N$ or $m = n$, one can increase the value of $H$ if necessary, so that $\theta^n (Hm)^{3m} \geq 1$, which  is greater or equal to $\dtv\lb\mu_{n,m},\eta_m\rb\leq1$.   
We conclude the proof of \eqref{eq:only need step 2}. Step 2 is complete.\\

{\bf \raggedleft \underline{Step 3.}}   In this step we show that the total variation distance bound \eqref{eq:only need step 2} continues to hold without the assumption that $q^lABCD\neq1$ for all $l\in\NN$. Note that the well-definedness of the measure $\eta_m$ (which we concluded from Step 1) does not have any issues, since it only involves $A,B,C$ and $q$.

Before we embark on the major component of Step 3, we present the following technical total variation bound of the Askey--Wilson signed measure $\pi_1$ by the mass of its largest atom:
\begin{proposition}\label{prop:bound on pi11}
   Assume $A,C\geq0$, $\max(A,C)>1$, $-1<B,D\leq 0$, $q\in[0,1)$ and $|(ABCD)_{\infty}|\leq1$. Assume also  $A/C\notin\{q^l:l\in\ZZ\}$ if $A,C\geq 1$. Then there exists a positive constant $L$ depending on $A,C$ and $q$ such that $$\lnor\pi_1\rnor\leq \frac{L}{1-BD}\lv\pi_1\lb\{y_0(1)\}\rb\rv.$$
\end{proposition}
\begin{remark}
    The above result will appear again as Proposition \ref{prop:bound on pi1} which will be proved in Appendix \ref{sec:total variation bounds} by bounding the norms of  atom masses in this Askey--Wilson signed measure.
\end{remark}

We return to Step 3 in the proof. We consider the case $q^jABCD=1$ for some $j\in\NN$.
We will use the continuity argument. We take a sequence of $D_k\in(-1,0)$, $k=1,2,\dots$ that converges to $D$, satisfying $q^lABCD_k\neq1$ for all $l\in\NN$ and for all $k$. Since $(ABCD)_{\infty}=0$, we can assume that $|(ABCD_k)_{\infty}|\leq 1$ for all $k$. Using the continuity of the stationary measure $\mu_n^{(A,B,C,D)}$ from Lemma \ref{lem:continuity}, we have:
\be\label{eq:1}\lim_{k\rightarrow\infty}\dtv\lb\mu^{(A,B,C,D_k)}_{n,m},\eta_m\rb=\dtv\lb\mu^{(A,B,C,D)}_{n,m},\eta_m\rb.\ee
Using the conclusion of Step 2, for each $k=1,2,\dots$ one can bound 
$\dtv\lb\mu^{(A,B,C,D_k)}_{n,m},\eta_m\rb$  
   by the corresponding RHS of \eqref{eq:bound TV distance in proof continuity}. We observe that most of the terms therein do not rely on $k$:
\begin{enumerate}
    \item [$\bullet$] The values of $\ep$ and $K$ were selected by Proposition \ref{prop:TV bound of AW signed measures1}, which only involve $A,B,C$ and $q$ and hence do not rely on $k$.
    \item [$\bullet$] The values of $\th=\frac{1+y^*_0(1)}{1+y_0(1)}$ and $r= \sup_{t\in I} \lb 1+t+2\sqrt{t}y_0(t)\rb $ depend solely on $\ep,A,C$  and $q$ (recall that $I=[1,1+\ep)$) and hence do not rely on $k$. 
\end{enumerate}
Therefore, in view of \eqref{eq:1}, we have:
\begin{equation*}
    \begin{split}
        \dtv\lb\mu^{(A,B,C,D)}_{n,m},\eta_m\rb\leq&\th^n r^m\frac{K^m}{(\ep/(2m))^{2m}} \lb\frac{2+\ep}{\ep/(2m)}\rb^m2^m  \\
        &\times\limsup_{k\rightarrow\infty}\frac{  \lnor\pi^{(A,B,C,D_k)}_1\rnor}{\lv\pi^{(A,B,C,D_k)}_1\lb\{y_0(1)\}\rb\rv-\th^n\lnor \pi^{(A,B,C,D_k)}_1\rnor}.
    \end{split}
\end{equation*} 
Using Proposition \ref{prop:bound on pi11} above, in view of our assumption $|(ABCD_k)_{\infty}|\leq 1$ and the fact that $1/(1-BD_k)$ is bounded by a uniform constant for all $k=1,2,\dots$, we conclude that there exists $N$ independent of $k$ such that for all $n\geq N$, the supremum limit in the RHS above is a finite positive number. Therefore the bound \eqref{eq:only need step 2}:
$$\dtv\lb\mu_{n,m},\eta_m\rb\leq\th^n (Hm)^{3m}$$
holds at $(A,B,C,D)$ for some positive number $H=H(A,B,C,D)$ for $n\geq N$ and $1\leq m<n$.  
Similar to the end of Step 2, for $n<N$ or $m=n$, one can increase the value of $H$ if necessary so that $\dtv\lb\mu_{n,m},\eta_m\rb \leq1$ can still be bounded by $\th^n (Hm)^{3m}$.
Step 3 is complete.

In view of Step 1, Step 2 and Step 3 above, we conclude the proof of Theorem \ref{thm:modulo particle hole}.
\end{proof}

In the rest of this subsection we will use Theorem \ref{thm:modulo particle hole} and the particle-hole duality to deduce Theorem \ref{thm:LD} and Theorem \ref{thm:HD}.   
    We recall the marginal distributions $\mu_{n|m}$ and  $\mu_{n,m}$ of $\mu_n$ and Remark \ref{rmk:mu}. In view of the particle-hole duality stated as Lemma \ref{lem:particle hole}, we have:
    \be\label{eq:mu and mu}
    \mu_{n|m}^{(A,B,C,D)}(\tau_1,\dots,\tau_m)=\mu_{n,m}^{(C,D,A,B)}(\tau_m,\dots,\tau_1)
    \quad\text{for any } \tau_1,\dots,\tau_m\in\{0,1\}.
    \ee

The next technical lemma shows that the total variation distance bound  $\th^n(Hm)^{3m}$ induces the correct growth rate of  $\{m_n\}_{n=1}^{\infty}$. Note that we only need  $p=3$ in this lemma.  
\begin{lemma}\label{lem:tech}
    For any $\th\in(0,1)$, $H>0$ and $p>0$, we denote $s=-\frac{1}{p}\log\th>0$. Then for any sequence $\{m_n\}_{n=1}^{\infty}$ satisfying $1\leq m_n\leq s\frac{n}{\log n}$ for $n=1,2,\dots$, we have $$\lim_{n\rightarrow\infty}\th^n(Hm_n)^{pm_n}=0.$$ 
\end{lemma}
\begin{proof} For any constant $R>0$, for sufficiently large $n$, we have:
\begin{multline*}
    (Rm_n)^{pm_n}\leq\lb\frac{Rsn}{\log n}\rb^{\frac{psn}{\log n}}=\lb\frac{Rs}{\log n}n\rb^{\frac{psn}{\log n}}\leq n^{\frac{psn}{\log n}}\\
=n^{\frac{n\log\lb\frac{1}{\th}\rb}{\log n}}=\lb n^{\frac{\log\lb\frac{1}{\th}\rb}{\log n}}\rb^n=\lb n^{\log_n\lb\frac{1}{\th}\rb}\rb^n=\lb\frac{1}{\th}\rb^n,
\end{multline*}
hence     $$\overline{\lim}_{n\rightarrow\infty}\th^n(Hm_n)^{pm_n}\leq\overline{\lim}_{n\rightarrow\infty}(H/R)^{pm_n}\leq (H/R)^p.$$ Taking $R\rightarrow\infty$ we conclude the proof.
\end{proof}

We now prove the main Theorem \ref{thm:LD} in the low density phase.
\begin{proof}[Proof of Theorem \ref{thm:LD}]
    Although this result is about the low density phase, we will use Theorem \ref{thm:modulo particle hole} in the high density phase and later use particle-hole duality.
    
    In the high density phase $A>1$, $A>C$, we   assume $C/A\notin\{q^l:l\in\ZZ_+\}$ if $C\geq1$ and analyze the measure $\eta_m$ introduced in Definition \ref{def:measure etam}. In view of the second statement in Lemma \ref{lem:cite},  we have $$P_{1,t_1}\lb y_0(1),\d x_1\rb\dots P_{t_{m-1},t_m}\lb x_{m-1},\d x_m\rb=\delta_{y_0(t_1)}(\d x_1)\dots\delta_{y_0(t_m)}(\d x_m).$$  We notice   $y_0(t)=\frac{1}{2}\lb A\sqrt{t}+ \frac{1}{A\sqrt{t}}\rb$ and therefore
    $$\frac{1+t_i+2\sqrt{t_i}y_0(t)}{2+2y_0(1)}=\frac{1+At_i}{1+A}.$$ Hence the RHS of \eqref{eq:def of eta} equals $\prod_{i=1}^n\frac{1+At_i}{1+A}$, using which we conclude that $\eta_m=\ber_m\lb\frac{A}{1+A}\rb$. 
    
    For the total variation distance bound, by Theorem \ref{thm:modulo particle hole} we have
    $$\dtv\lb\mu_{n,m},\ber_m\lb\frac{A}{1+A}\rb\rb\leq\th^n (Hm)^{3m}.$$ 
    In view of particle-hole duality \eqref{eq:mu and mu}, in the low density phase $C>1$, $C>A$, assuming that $A/C\notin\{q^l:l\in\ZZ_+\}$ if $A\geq1$, we have that for any $1\leq m\leq n$,
    \be\label{eq:rvwe}\dtv\lb\mu_{n|m},\ber_m\lb\frac{1}{1+C}\rb\rb\leq\th^n (Hm)^{3m}. \ee 
    We conclude the total variation bound \eqref{eq:bound LD} in Theorem \ref{thm:LD}. The last statement in the theorem (about the sequence $\{m_n\}$) follows directly from this bound combining with technical Lemma \ref{lem:tech}.
    
    For fixed $m$, we note that the RHS of \eqref{eq:rvwe} converges to $0$ as $n\rightarrow\infty$. In particular, the sequence of measures $\mu_{n|m}$ converges to $\ber_m\lb\frac{1}{1+C}\rb$. Next we show this convergence holds true without the assumption $A/C\notin\{q^l:l\in\ZZ_+\}$ for $A\geq1$. 
    We fix the jump rates $\alpha,\beta,\gamma,q$ and choose $\delta'>\delta$ so that $A':=\phi_+(\beta,\delta')$ satisfy $A'/C\notin\{q^l:l\in\ZZ_+\}$. 
    In view of the stochastic sandwiching Lemma \ref{lem:sandwiching}, we have 
     \be \label{eq:wre}\EE_{\mu_{n|m}^{(A',B,C,D)}}\le\prod_{i=1}^mt_i^{\tau_i}\re\geq\EE_{\mu_{n|m}^{(A,B,C,D)}}\le\prod_{i=1}^mt_i^{\tau_i}\re,\ee
     for any $t_1,\dots,t_m\geq1$.
    As $n\rightarrow\infty$, we know that $\mu_{n|m}^{(A',B,C,D)}$ converges to $\ber_m\lb\frac{1}{1+C}\rb$, therefore the  LHS of \eqref{eq:wre} converges to $$\EE_{\ber_m\lb\frac{1}{1+C}\rb}\le\prod_{i=1}^mt_i^{\tau_i}\re=\prod_{i=1}^m\frac{t_i+C}{1+C}.$$   We conclude
    $$ \prod_{i=1}^m\frac{t_i+C}{1+C}\geq\limsup_{n\rightarrow\infty}\EE_{\mu_{n|m}^{(A,B,C,D)}}\le\prod_{i=1}^mt_i^{\tau_i}\re.$$ 
    By the same reason, the opposite inequality holds true. Therefore the convergence holds true for $(A,B,C,D)$.     We conclude that on the entire high density phase, the sequence of measures $\mu_{n}$ converges weakly to the product Bernoulli measure with density $\frac{1}{1+C}$ on $\{0,1\}^{\ZZ_+}$. The proof is concluded.
\end{proof}

Before we provide the proof of Theorem \ref{thm:HD}, we present a special symmetry of multi-dimensional Askey--Wilson signed measures known as the `time reversal', which will be needed in the proof.
\begin{proposition}\label{prop:time reversal1}
    Assume $q\in[0,1)$ and $A,B,C,D\in\RR$.   We have:
    \be 
    \pi_{t_1,\dots,t_m}^{(A,B,C,D)}\lb\d x_1,\dots,\d x_m\rb=
    \pi_{1/t_m,\dots,1/t_1}^{(C,D,A,B)}\lb\d x_m,\dots,\d x_1\rb
    \ee
    for any $0<t_1\leq\dots\leq t_m$ for which the multi-dimensional Askey--Wilson signed measures on both sides of this identity are well-defined. 
\end{proposition}
\begin{remark}
    The above result will reappear  as Proposition \ref{prop:time reversal} and will be proved in Appendix \ref{sec:time reversal}.
\end{remark}

We now prove the main  Theorem \ref{thm:HD} in the high density phase.
\begin{proof}[Proof of Theorem \ref{thm:HD}]
    Although this result is about the high density phase, we will use Theorem \ref{thm:modulo particle hole} in the low density phase and later use particle-hole duality to conclude the proof.

    In the low density phase $C>1$, $C>A$, we  assume $A/C\notin\{q^l:l\in\ZZ_+\}$ if $A\geq1$ and analyze the measure $\eta_m$ introduced in Definition \ref{def:measure etam}.  Notice $y_0(t)=\frac{1}{2}\lb\frac{C}{\sqrt{t}}+\frac{\sqrt{t}}{C} \rb$ and  hence
    $$P_{1,t_1}^{(A,B)}\lb y_0(1),\d x_1\rb=\nu\lb\d x_1; A\sqrt{t_1}, B\sqrt{t_1}, C/\sqrt{t_1}, 1/\lb C\sqrt{t_1} \rb\rb
    =\pi_{t_1}^{(A,B,C,1/C)}(\d x_1).$$
Therefore the signed measure on the RHS of \eqref{eq:def of eta} equals:
\begin{multline*}
     P_{1,t_1}^{(A,B)}\lb y_0(1),\d x_1\rb\dots P^{(A,B)}_{t_{m-1},t_m}\lb x_{m-1},\d x_m\rb\\
     =\pi_{t_1}^{(A,B,C,1/C)}(\d x_1)P^{(A,B)}_{t_{1},t_2}\lb x_{1},\d x_2\rb\dots P^{(A,B)}_{t_{m-1},t_m}\lb x_{m-1},\d x_m\rb \\
      =\pi^{(A,B,C,1/C)}_{t_1,\dots,t_m}\lb\d x_1,\dots,\d x_m\rb,
\end{multline*}
Since $2+2y_0(1)= (1+C)^2/C$, by \eqref{eq:def of eta} we conclude that, for any $1\leq t_1\leq\dots\leq t_m<1+\ep$,
\be  \label{eq:eta in low density phase}
    \EE_{\eta_m}\le\prod_{i=1}^mt_i^{\tau_i}\re=\frac{C^m}{(1+C)^{2m}}\int_{\RR^m}\prod_{i=1}^m\lb1+t_i+2\sqrt{t_i}x_i\rb
        \pi^{(A,B,C,1/C)}_{t_1,\dots,t_m}\lb\d x_1,\dots,\d x_m\rb.
    \ee  

We define the measure $\la_m$ (whose existence is claimed in Theorem \ref{thm:HD}) as the particle-hole duality of $\eta_m$, i.e. in the high density phase $A>1$, $A>C$, assuming $C/A\notin\{q^l:l\in\ZZ_+\}$ if $C\geq1$, we define:
$$\la_m^{(A,B,C,D)}(\tau_1,\dots,\tau_m)=\eta_m^{(C,D,A,B)}(1-\tau_m,\dots,1-\tau_1)\quad\text{for all}\quad\tau_1,\dots,\tau_m\in\{0,1\}.$$
Therefore, for any $\frac{1}{1+\ep}<t_1\leq\dots\leq t_m\leq 1$, we have:
\begin{equation*}
    \begin{split}
        &\EE_{\la_m^{(A,B,C,D)}}\le\prod_{i=1}^mt_i^{\tau_i}\re =\EE_{\eta_m^{(C,D,A,B)}}\le\prod_{i=1}^mt_{m+1-i}^{1-\tau_i}\re \\
        & =t_1\dots t_m\EE_{\eta_m^{(C,D,A,B)}}\le\prod_{i=1}^m\lb\frac{1}{t_{m+1-i}}\rb^{\tau_i}\re\\
        &=t_1\dots t_m\frac{A^m}{(1+A)^{2m}}\int_{\RR^m}\prod_{i=1}^m\lb1+\frac{1}{t_{m+1-i}}+2\sqrt{\frac{1}{t_{m+1-i}}}x_i\rb\pi^{(C,D,A,1/A)}_{\frac{1}{t_m},\dots,\frac{1}{t_1}}\lb\d x_1,\dots,\d x_m\rb\\
        &=\frac{A^m}{(1+A)^{2m}}\int_{\RR^m}\prod_{i=1}^m\lb1+t_{m+1-i}+2\sqrt{t_{m+1-i}}x_i\rb\pi^{(A,1/A,C,D)}_{t_1,\dots,t_m}\lb\d x_m,\dots,\d x_1\rb\\
        &=\frac{A^m}{(1+A)^{2m}}\int_{\RR^m}\prod_{i=1}^m\lb1+t_i+2\sqrt{t_i}x_i\rb\pi^{(A,1/A,C,D)}_{t_1,\dots,t_m}\lb\d x_1,\dots,\d x_m\rb,
        \end{split}
\end{equation*} 
where we have used \eqref{eq:eta in low density phase} for $1\leq1/t_m\leq\dots\leq1/t_1<1+\ep$ and the time reversal symmetry of Askey--Wilson signed measures (Proposition \ref{prop:time reversal1}). The above identity coincides with \eqref{eq:lambda}  in Definition \ref{def:HD limit measure}. By particle-hole duality, the total variation distance bound \eqref{eq:bound HD} in Theorem \ref{thm:HD} directly follows from the bound \eqref{eq:TV bound final} in Theorem \ref{thm:modulo particle hole}. The last statement in the theorem (about the sequence $\{m_n\}$) follows directly from this bound combining with the technical Lemma \ref{lem:tech}. We conclude the proof. 
\end{proof}

\subsection{Proof of Proposition \ref{prop:product Bernoulli}}\label{sec:proof of product Bernoulli}
In this subsection we prove Proposition \ref{prop:product Bernoulli}, which states that the limiting measure $\lambda$ on $\{0,1\}^{\mathbb{Z}_+}$ in the high density phase, as introduced in Definition \ref{def:HD limit measure}, is a product Bernoulli measure if and only if $AC = 1$, in which case it has density $A/(1 + A)$.
\begin{proof}[Proof of Proposition \ref{prop:product Bernoulli}]
Assume that the measure $\lambda$ on $\{0,1\}^{\mathbb{Z}_+}$ is product Bernoulli with density $\rho\in[0,1]$, then each $\la_m$ on $\{0,1\}^m$ for $m\in\ZZ_+$ is product Bernoulli with density $\rho$. Take $t_1=\dots=t_m=t\in(1-\ep,1)$ in \eqref{eq:lambda} we have
\be \label{eq:ergareger}
\begin{split}
\EE_{\la_m}\le t^{\tau_1+\dots+\tau_m}\re&=(t\rho+1-\rho)^m\\
    &=\frac{A^m}{(1+A)^{2m}}\int_{\RR^m}\prod_{i=1}^m\lb1+t+2\sqrt{t}x_i\rb
        \pi^{(A,1/A,C,D)}_{t,\dots,t}\lb\d x_1,\dots,\d x_m\rb\\
    &=\frac{A^m}{(1+A)^{2m}}\int_{\RR}\lb1+t+2\sqrt{t}x\rb^m\pi_t^{(A,1/A,C,D)}(\d x).
\end{split}
\ee  
Recall that
$ \pi^{(A,1/A,C,D)}_t(\d x)=\nu\lb \d x;A\sqrt{t}, \sqrt{t}/A,C/\sqrt{t},D/\sqrt{t}\rb$. Since $A>C\geq0$, $A>1$  and $D\in(-1,0]$, one can shrink the value of $\ep>0$ if necessary, so that for all $t\in(1-\ep,1)$, we have  $A\sqrt{t}>1$, $A\sqrt{t}>C/\sqrt{t}\geq0$, $\sqrt{t}/A\in[0,1)$ and $D/\sqrt{t}\in(-1,0]$. Therefore the signed measure $\pi_t(\d x)$ only has atoms $>1$, and the largest atom is
$y_0(t)=\frac{1}{2}\lb A\sqrt{t}+\frac{1}{A\sqrt{t}}\rb$. Using formula \eqref{eq: p_0} one can check that $\pi_t(\{y_0(t)\})\neq0$. We denote by $y_1^*(t)$ the second largest atom of $\pi_t(\d x)$ with nonzero mass, if it exists; otherwise, set $y_1^*(t) = 1$. 
Then $y_1^*(t)<y_0(t)$, and that $\pi_t(\d x)$ is supported on $U_t\subset \{ y_0(t)\}\cup[-1,y_1^*(t)]$. Therefore we have
\begin{multline}\label{eq:verererg}
\int_{\RR}\lb1+t+2\sqrt{t}x\rb^m\pi_t(\d x)=\\\int_{\{y_0(t)\}}\lb1+t+2\sqrt{t}x\rb^m\pi_t(\d x)+\int_{-1}^{y_1^*(t)}\lb1+t+2\sqrt{t}x\rb^m\pi_t(\d x).
\end{multline}
The first term on the RHS equals $(1+t+2\sqrt{t}y_0(t))^m \pi_t(\{y_0(t)\})$. The absolute value of the second term on the RHS is bounded by $(1+t+2\sqrt{t}y_1^*(t))^m\lnor\pi_t\rnor$, which has a lower order as $m\rightarrow\infty$ since $1<1+t+2\sqrt{t}y_1^*(t)<1+t+2\sqrt{t}y_0(t)$. Therefore by \eqref{eq:ergareger},  
\be\label{eq:vegerv} 
\EE_{\la_m}\le t^{\tau_1+\dots+\tau_m}\re=(t\rho+1-\rho)^m\sim \frac{A^m}{(1+A)^{2m}}\lb 1+t+2\sqrt{t}y_0(t)\rb^m\pi_t(\{y_0(t)\}),
\ee
where we write $f(m)\sim g(m)$ if $f(m)/g(m)\rightarrow1$ as $m\rightarrow\infty$. The above implies
\[
t\rho+1-\rho=\frac{A}{(1+A)^{2}}\lb1+t+2\sqrt{t}y_0(t)\rb=\frac{1+At}{1+A}
\]
hence $\rho=A/(1+A)$, and also $\pi_t(\{y_0(t)\})=1$. Therefore, we conclude that \eqref{eq:vegerv} is an equality, and hence the second term on the RHS of \eqref{eq:verererg} is equal to zero. If $y_1^*(t)>1$ is an atom, then by a similar analysis as above, this second term
\[
0=\int_{-1}^{y_1^*(t)}\lb1+t+2\sqrt{t}x\rb^m\pi_t(\d x)\sim \lb 1+t+2\sqrt{t}y_1^*(t)\rb^m\pi_t(\{y_1^*(t)\}),
\]
which is a contradiction since $\pi_t(\{y_1^*(t)\})\neq0$ by our assumption. Therefore we conclude that $\pi_t(\d x)$ only has a single atom at $y_0(t)$ with mass $1$. Observe that the continuous part density \eqref{eq:continuous part density} of an Askey--Wilson signed measure has fixed sign over $x\in[-1,1]$, we conclude that the continuous part of $\pi_t(\d x)$ is constantly zero. Using formula \eqref{eq:continuous part density} we have
\[
\lb q,t,AC,AD,C/A,D/A,CD/t\rb_{\infty}=0.
\]
Since $t\in(1-\ep,1)$, $A>C\geq0$ and $D\in(-1,0]$, we have $ACq^k=1$ for some $k\in\NN$.  

We first assume $k\in\ZZ_+$. We choose $t\in(1-\ep,1)$ satisfying $A\sqrt{t}q\neq1$. Then by $ACq^k=1$, we have either $A\sqrt{t}q>1$ or $C/\sqrt{t}>1$. When $A\sqrt{t}q>1$, by formula \eqref{eq: p_j}, the second largest atom generated by $A\sqrt{t}$, i.e., $\frac{1}{2}\lb A\sqrt{t}q+\frac{1}{A\sqrt{t}q}\rb$, has nonzero mass. When $C/\sqrt{t}>1$, by \eqref{eq: p_0}, the largest atom generated by $C/\sqrt{t}$, i.e., $\frac{1}{2}\lb \frac{C}{\sqrt{t}}+\frac{\sqrt{t}}{C}\rb$, has nonzero mass. This contradicts to the fact that  $\pi_t(\d x)$ has only one atom $y_0(t)$ with nonzero mass. 

We conclude that $k=0$, i.e., $AC=1$. When $AC=1$, by a computation, the multi-dimensional Askey--Wilson signed measure  $\pi^{(A,1/A,C,D)}_{t_1,\dots,t_m}\lb\d x_1,\dots,\d x_m\rb$ on the RHS of \eqref{eq:lambda} is a point mass at $x_i=y_0(t_i)$ for $i=1,\dots,m$. Using \eqref{eq:lambda} we conclude that $\la_m$ is the product Bernoulli measure on $\{0,1\}^m$ with density $A/(1+A)$. Therefore $\la$ is the product Bernoulli measure on $\{0,1\}^{\ZZ_+}$ with the same density. We conclude the proof.
\end{proof}

  \begin{appendix}
 
\section{Bounding the total variation distance by generating functions}
In this appendix, we provide a general bound of the total variation distance between two  probability measures on $\{0,1\}^m$ by the value of the difference of their joint generating functions at certain points. While it is possible that this inequality is already known, we have been unable to locate it in the previous literature.
\label{sec:bounding TV distance by generating functions}
\begin{proposition}\label{prop:bounding distance two prob measures}
    Let $\k$ and $\k'$ be probability measures on $\{0,1\}^m$. Then for any set of numbers $0<t_{i,0}<t_{i,1}$ for $i=1,\dots,m$, we have:
    \be\label{eq:bound TV distance} 
    \dtv(\k,\k')
    \leq \frac{1}{2}\prod_{i=1}^m\frac{1+t_{i,1}}{t_{i,1}-t_{i,0}}\sum_{\up_1,\dots,\up_m\in\{0,1\}}
    \lv\EE_{\k}\le\prod_{i=1}^mt_{i,\up_i}^{\tau_i}\re - \EE_{\k'}\le\prod_{i=1}^mt_{i,\up_i}^{\tau_i}\re\rv,\ee
    where $\tau_i\in\{0,1\}$ is the occupation variable on the site $i$, for $i=1,\dots,m$.
\end{proposition}
\begin{proof}
    This result can be seen as a simple corollary of the following inequality:
    For any degree-$1$ multivariate polynomial $G$ with real coefficients:
$$G(x_1,\dots,x_m)=\sum_{\up_1,\dots,\up_m\in\{0,1\}}g_{\up_1,\dots,\up_m}x_1^{\up_1}\dots x_m^{\up_m},$$
and for any set of numbers $0<t_{i,0}<t_{i,1}$, $i=1,\dots,m$, we have:
\be\label{eq:polynomial inequality in prop}\sum_{\up_1,\dots,\up_m\in\{0,1\}}|g_{\up_1,\dots,\up_m}|\leq\prod_{i=1}^m\frac{1+t_{i,1}}{t_{i,1}-t_{i,0}} \sum_{\up_1,\dots,\up_m\in\{0,1\}}|G(t_{1,\up_1},\dots,t_{m,\up_m})|.\ee
Specifically, \eqref{eq:bound TV distance} can be seen by taking $$g_{\tau_1,\dots,\tau_m}=\k(\tau_1,\dots,\tau_m)-\k'(\tau_1,\dots,\tau_m)$$ for any $\tau_1,\dots,\tau_m\in\{0,1\}$.

We now prove \eqref{eq:polynomial inequality in prop} by induction on $m$. 
When $m=1$, we have $G(x_1)=g_0+g_1x_1$, where
$$ 
g_0=\frac{-t_{1,0}G(t_{1,1})+t_{1,1}G(t_{1,0})}{t_{1,1}-t_{1,0}},\quad
g_1=\frac{G(t_{1,1})-G(t_{1,0})}{t_{1,1}-t_{1,0}}.$$
Therefore 
$$|g_0|+|g_1|\leq \frac{1+t_{1,1}}{t_{1,1}-t_{1,0}}\lb|G(t_{1,0})|+|G(t_{1,1})|\rb.$$
Suppose $m\geq 2$, and \eqref{eq:polynomial inequality in prop} holds for the $m-1$ case. 
 We write
   $$G(x_1,\dots,x_m)=G_0(x_1,\dots,x_{m-1})+G_1(x_1,\dots,x_{m-1})x_m,$$
   where for $j\in\{0,1\}$,
    $$G_j(x_1,\dots,x_{m-1})=\sum_{\up_1,\dots,\up_m\in\{0,1\}}g_{\up_1,\dots,\up_{m-1},j}x_1^{\up_1}\dots x_{m-1}^{\up_{m-1}}.$$ 
By the inequality for $m=1$, for any $0<t_{m,0}<t_{m,1}$ and real values of $x_1,\dots,x_{m-1}$ we have:
\begin{multline*}
    |G_0(x_1,\dots,x_{m-1})|+|G_1(x_1,\dots,x_{m-1})|\\
    \leq \frac{1+t_{m,1}}{t_{m,1}-t_{m,0}}
\lb|G(x_1,\dots,x_{m-1},t_{m,0})|+|G(x_1,\dots,x_{m-1},t_{m,1})|\rb.
\end{multline*} 
We sum the above inequality over all the $2^{m-1}$ possibilities $x_i\in\{t_{i,0},t_{i,1}\}$, $1\leq i\leq m-1$ and get: 
\begin{multline}
    \label{eq:in proof of TV bound}\sum_{j\in\{0,1\}}\sum_{\up_1,\dots,\up_{m-1}\in\{0,1\}}|G_j(t_{1,\up_1},\dots t_{m-1,\up_{m-1}})|\\\leq \frac{1+t_{m,1}}{t_{m,1}-t_{m,0}}\sum_{\up_1,\dots,\up_{m}\in\{0,1\}}|G(t_{1,\up_1},\dots t_{m,\up_{m}})|.
\end{multline} 
By the induction hypothesis, for each $j\in\{0,1\}$,
    $$\sum_{\up_1,\dots,\up_{m-1}\in\{0,1\}}|g_{\up_1,\dots,\up_{m-1},j}|
    \leq\prod_{i=1}^{m-1}\frac{1+t_{i,1}}{t_{i,1}-t_{i,0}}\sum_{\up_1,\dots,\up_{m-1}\in\{0,1\}}|G_j(t_{1,\up_1},\dots t_{m-1,\up_{m-1}})|.$$
We sum the above inequality over $j\in\{0,1\}$ and get: 
\begin{equation*}
\begin{split}
    \sum_{\up_1,\dots,\up_m\in\{0,1\}}|g_{\up_1,\dots,\up_m}|&\leq \prod_{i=1}^{m-1}\frac{1+t_{i,1}}{t_{i,1}-t_{i,0}}\sum_{j\in\{0,1\}}\sum_{\up_1,\dots,\up_{m-1}\in\{0,1\}}|G_j(t_{1,\up_1},\dots t_{m-1,\up_{m-1}})|\\
    &\leq\prod_{i=1}^m\frac{1+t_{i,1}}{t_{i,1}-t_{i,0}}\sum_{\up_1,\dots,\up_m\in\{0,1\}}|G(t_{1,\up_1},\dots t_{m,\up_m})|,
\end{split}
\end{equation*} 
where  the last step uses \eqref{eq:in proof of TV bound}.
Hence inequality \eqref{eq:polynomial inequality in prop} is proved for $m$ case. We conclude the proof. 
\end{proof}
 
\section{Total variation bounds of Askey--Wilson signed measures}
\label{sec:total variation bounds}
In this appendix, we establish two results which bound the total variations of certain Askey--Wilson signed measures. 
\begin{proposition}\label{prop:TV bound of AW signed measures}
Assume $A,C\geq0$, $-1<B,D\leq 0$ and $q\in[0,1)$.
    Assume also  that $q^lABCD\neq1$ for all $l\in\NN$ and that $A/C\notin\{q^l:l\in\ZZ\}$ if $A,C\geq 1$. Then there exist positive constants $K\geq1$ and $\ep$ depending on $A,B,C$ and $q$ but not on $D$, such that for any $s<t$ in $I=[1,1+\ep)$ and $x\in U_s$, the total variation of the Askey--Wilson signed measure $P_{s,t}(x,\d y)$ is bounded from above by $\frac{K}{(t-s)^2}$. 
\end{proposition}
\begin{remark}
    As   mentioned in Remark \ref{rmk:tv bound}, the above total variation bound is more refined than the one utilized in \cite{wang2023askey}. In particular, the total variation bound of $P_{s,t}(x,\mathrm{d}y)$ provided by \cite[Proposition A.1]{wang2023askey} can be derived as a simple corollary of the above result, but the converse does not hold. As we will see in the proof, more subtle estimates of the atom masses are needed to establish this result.
\end{remark}
Before we offer the proof of Proposition \ref{prop:TV bound of AW signed measures}, we first prove a lemma bounding the total variation of an Askey--Wilson signed measure by the supremum of all the atom masses:
\begin{lemma}\label{lem:bound TV}
    Assume $q\in[0,1)$ and $(a,b,c,d)\in\Omega$, where $\Omega$ is introduced in Definition \ref{def:AW}. Then
    \be\label{eq:l1}\lnor\nu\lb\d y;a,b,c,d\rb\rnor\leq1+2\card\lb F(a,b,c,d)\rb \sup_{y_j^{\eee}\in F(a,b,c,d)}|p_j^{\eee}(a,b,c,d)|,\ee
    where   $\lnor\nu\rnor$   denotes the total variation of a signed measure $\nu$ and $\card(\mathcal{S})$  denotes the cardinality of a set $\mathcal{S}$. We recall that $p_j^{\eee}(a,b,c,d)$ denotes the mass of the atom $y_j^{\eee}$.
\end{lemma}
\begin{proof}
We recall from Section \ref{subsec:Backgrounds} that for any Askey--Wilson signed measure, the continuous part density on $(-1,1)$ must be either constantly positive or  constantly negative. Furthermore, the total mass is always equal to $1$. Therefore we have,
    $$\lnor\nu\lb\d y;a,b,c,d\rb\rnor \leq\sum_{y_j^{\eee}\in F(a,b,c,d)}|p_j^{\eee}|+\int_{-1}^1|f(y;a,b,c,d)|\d y
    \leq 1+2\sum_{y_j^{\eee}\in F(a,b,c,d)}|p_j^{\eee}|.$$
The RHS above can be further bounded by the RHS of \eqref{eq:l1}.
\end{proof}
 
\begin{proof}[Proof of Proposition \ref{prop:TV bound of AW signed measures}]
We assume that $A,C>0$ in the proof. Otherwise we are in the fan region $AC<1$, in which the Askey--Wilson signed measures are actually probability measures, whose total variation equal $1$.
We   first prove that there exists $K\geq1$ and $\ep>0$ depending on $A,B,C,D$ and $q$, such that for any $s<t$ in $I=[1,1+\ep)$ and $x\in U_s$,
\be\label{eq:only need appendix B} \lnor P_{s,t}(x,\d y)\rnor \leq\frac{K}{(t-s)^2}.\ee
At the end of the proof we will show that the above constants $K$ and $\ep$ can be chosen as independent of $D$.

    We start by choosing $\ep>0$ according to (the proof of) Theorem \ref{thm:characterization}. 
    For any $s<t$ in $I=[1,1+\ep)$ and $x\in U_s$,
    we will investigate the Askey--Wilson signed measure $P_{s,t}^{(A,B)}(x,\d y)=\nu\lb\d y;a,b,c,d\rb$, where
    \be\label{eq:abcd}a=A\sqrt{t},\quad b=B\sqrt{t},\quad c=\sqrt{\frac{s}{t}}\lb x+\sqrt{x^2-1}\rb,\quad d=\sqrt{\frac{s}{t}}\lb x-\sqrt{x^2-1}\rb.\ee   
    Notice that the norms of $a,b,c,d$  are uniformly bounded by a finite constant. Hence the total number of atoms in $\nu\lb\d y;a,b,c,d\rb$ is also uniformly bounded. Here and below, a uniform constant means a constant that only depends on $A,B,C,D$ and $q$.
    In view of Lemma \ref{lem:bound TV} we only need to bound the supremum of norms of all atom masses by a uniform constant over $(t-s)^2$.
    
    We look at the   atom masses: If $|aq^k|\geq1$ for $k\in\NN$, then:
\be\label{eq:formulas for atom masses in proof}
\begin{split}
    p_0^\aa&=\frac{(a^{-2},bc,bd,cd)_\infty}{(b/a,c/a,d/a,abcd)_\infty},\\
    p_k^\aa&=\frac{(a^{-2},bc,bd,cd)_\infty}{(b/a,c/a,d/a,abcd)_\infty}\frac{q^k(1-a^2q^{2k})(a^2,ab,ac,ad)_k}{(q)_k(1-a^2)a^{4k}\prod_{l=1}^{k}\lb(b/a-q^l)(c/a-q^l)(d/a-q^l)\rb},\quad k\geq 1.
\end{split}
\ee 
The masses for atoms generated by $\eee\in\{\cc,\dd\}$ are given by similar formulas with $a$ and $\e$ swapped. 
Note that $\bb$ does not generate atoms since $b=B\sqrt{t}\in(-1,0]$. One can observe that the numerators of all   atom masses are uniformly bounded from above. In the denominator, the term $(abcd)_{\infty}$ is uniformly bounded away from $0$, since $abcd$ equals $ABt\leq 0$. Other three terms $(1-\e^2)$, $\e^{4k}$ and $(q)_k$ are uniformly bounded away from $0$, since we have $|\e q^k|\geq 1$ and that $k\geq 1$ is bounded from above.  
Therefore, for any atom $y_k^{\eee}$ generated by $\eee\in\{\aa,\cc,\dd\}$,
\be\label{eq:in proof bound atom masses}
|p_k^{\eee}(a,b,c,d)|\leq\const\times\prod_{\f\in\{a,b,c,d\}\setminus\{\e\}}\prod_{l=-k}^{\infty}\frac{1}{  | 1-q^l\f/\e |},\quad k\in\NN, \quad|\e q^k|\geq 1.
\ee
Here and below, we use $\const$ to denote a uniform positive constant. The specific value of this constant may vary in different bounds.

The bound \eqref{eq:only need appendix B} then follows from the following claim:  
\begin{claim}\label{claim:only need}
    For any $\eee \in\{\aa,\cc,\dd\}$ such that $|\e|\geq1$, we have $\e\geq1$.  Consider the following set:
$$\mathcal{A}_{k}^{\eee,\fff}:=\left\{\left|1-q^l\frac{\f}{\e}\right|: l\geq -k\right\},\quad k\in\NN,\quad \e q^k\geq 1,\quad\fff\in\{\aa,\bb,\cc,\dd\}\setminus\{\eee\}.$$
For fixed $\eee$, $\fff$ and $k$, all the elements in the set $\mathcal{A}_{k}^{\eee,\fff}$ are uniformly bounded away from $0$ except for possibly one `exceptional' element. For this exceptional element $|1-q^l\f/\e|$, we have that either $\fff=\aa$ or $\fff=\cc$, and that it can be lower bounded by $(t-s)$ times a uniform positive constant.
\end{claim}

Notice that, although the RHS of \eqref{eq:in proof bound atom masses} is an infinite product, we do not need to worry about its tail, because it  is uniformly bounded as the inverse of a $q$-Pochhammer symbol. 
If Claim \ref{claim:only need} holds, then for any fixed $\eee\in\{\aa,\cc,\dd\}$ and $k\in\NN$, 
there are at most two pairs of $\fff\in\{\aa,\cc\}$ and $l\geq-k$ that produce exceptional elements $|1-q^l\f/\e|$.
Hence in view of \eqref{eq:in proof bound atom masses}, the atom mass $|p_k^{\eee}(a,b,c,d)|$ is bounded by a uniform constant times $1/(t-s)^2$. Therefore in view of Lemma \ref{lem:bound TV}, bound \eqref{eq:only need appendix B} holds. 

The proof of Claim \ref{claim:only need} necessitates a detailed analysis of all the possible values of $\f/\e$, which is provided in the following claim:

\begin{claim}\label{claim}
    For any $\e \in\{a,c,d\}$ satisfying $|\e|\geq1$, we have $\e\geq1$.  Consider the value of $\f/\e$ for $\f\in\{a,b,c,d\}\setminus\{\e\}$ with $\f>0$. 
    Either   $|\f/\e|$ can be expressed as an element in  
     $$\left\{\frac{\sqrt{s}}{At},\frac{s}{t},\frac{t}{s},\frac{1}{A^2t},\frac{C}{At},\frac{At}{C}\right\}$$ 
    times $q^r$ for some $r\in\ZZ$, or $\e=c$, $\f=d$ and $\f/\e\leq q^{2k}s/t$, where $k\in\NN$ satisfies $\e q^k\geq 1$.
\end{claim}
\begin{proof}[Proof of Claim \ref{claim}]
We recall that we always have $a>0$ and $b\in(-1,0]$.
Similar to the proof of Theorem \ref{thm:characterization}, we consider the following three cases depending on $x\in U_s$:
    \begin{enumerate}
    \item [Case 1.] Let $x\in[-1,1]$ then we have $a=A\sqrt{t}$, $b=B\sqrt{t}$, $c=\sqrt{\frac{s}{t}}\lb x+\sqrt{x^2-1}\rb$ and  $d=\sqrt{\frac{s}{t}}\lb x-\sqrt{x^2-1}\rb$. We have $|c|,|d|<1$ and $|c/a|=|d/a|=\sqrt{s}/(At)$.   Claim \ref{claim} holds.
\item [Case 2.] Let $x=\frac{1}{2}\lb q^jA\sqrt{s}+\lb q^jA\sqrt{s}\rb^{-1}\rb$ for $j\in\NN$ and $q^jA\sqrt{s}>1$.
Then $a=A\sqrt{t}$, $b=B\sqrt{t}$, $c=q^jAs/\sqrt{t}>0$ and $d=1/\lb q^jA\sqrt{t}\rb\in(0,1)$. We have $c/a=q^js/t$, $d/a=1/(q^jA^2t)$, $a/c=t/(q^js)$ and $d/c=1/(q^{2j}A^2s)$. If $\e=c$ then $q^kc=q^{k+j}As/\sqrt{t}\geq 1$, and hence $d/c\leq q^{2k}s/t$.  Therefore Claim \ref{claim} holds.
    \item [Case 3.] Let $x=\frac{1}{2}\lb q^jC/\sqrt{s}+\lb q^jC/\sqrt{s}\rb^{-1}\rb$ for $j\in\NN$ and $q^jC/\sqrt{s}>1$. 
    Then $a=A\sqrt{t}$, $b=B\sqrt{t}$, $c=q^jC/\sqrt{t}>0$ and $d=s/(q^jC\sqrt{t})\in(0,1)$. We have $c/a=q^jC/(At)$, $d/a=s/(q^jACt)$, $a/c=At/(q^jC)$ and $d/c=s/(q^{2j}C^2)$. If $\e=c$ then $q^kc=q^{k+j}C/\sqrt{t}\geq 1$, and hence $d/c\leq q^{2k}s/t$.   Therefore Claim \ref{claim} holds.
\end{enumerate}

We conclude the proof of Claim \ref{claim}.
\end{proof}

We now begin to prove Claim \ref{claim:only need}. 

\begin{proof}[Proof of Claim \ref{claim:only need}]
The fact that $\e\geq1$ has been proven by Claim \ref{claim}. To prove the rest of the statements, we use Claim \ref{claim} to divide the proof into the following three cases:
\begin{enumerate}
    \item [Case 1.] When $\f\leq 0$, we have $|1-q^l\f/\e|\geq 1$. In this case, every element in $\mathcal{A}_{k}^{\eee,\fff}$ is lower bounded by $1$.
    \item [Case 2.] When we have
    $$\frac{\f}{\e}\in\left\{\frac{\sqrt{s}}{At},\frac{s}{t},\frac{t}{s},\frac{1}{A^2t},\frac{C}{At},\frac{At}{C}\right\}\times\{q^r:r\in\ZZ\},$$ 
    where we denote $A\times B:=\{ab:a\in A, b\in B\}$ for  $A,B\subset\RR$. 
    We factorize the numbers above as products of the parts that involve $A$ and $C$ and the other parts that involve $s$ and $t$ (which can be $1$):
    \begin{multline*}
        \frac{\sqrt{s}}{At}=\frac{1}{A}\times\frac{\sqrt{s}}{t},\quad\frac{s}{t}=1\times\frac{s}{t},\quad \frac{t}{s}=1\times\frac{t}{s},\quad\\\frac{1}{A^2t}=\frac{1}{A^2}\times\frac{1}{t},\quad\frac{C}{At}=\frac{C}{A}\times\frac{1}{t},\quad\frac{At}{C}=\frac{A}{C}\times t. \end{multline*}
    
    One can shrink the value of $\ep>0$ if necessary, to ensure that there exists an uniform $\k>0$ such that for each $x\in\{1/2,1,-1\}$ and $y=\pm 1$ and for any $1\leq s<t<1+\ep$      we have $|1-q^ius^xt^y|>\k$ when: \textbf{(1)}   $i\in\ZZ\setminus\{0\}$ and $u=1$  or \textbf{(2)} $i\in\ZZ$ and
    $$u\in\left\{\frac{1}{A},\frac{1}{A^2},\frac{C}{A},\frac{A}{C}\right\}\setminus\{q^j:j\in\ZZ\}.$$

    Therefore, in the cases when 
    $$\frac{\f}{\e}\in \left\{\frac{s}{t},\frac{t}{s}\right\}\times\{q^r:r\in\ZZ\}:$$
Denote $\frac{\f}{\e}\in  \left\{q^h\frac{s}{t},q^h\frac{t}{s}\right\}$, $h\in\ZZ$. Then if
$h+l\in\ZZ\setminus\{0\}$ we have $\left|1-q^l\frac{\f}{\e}\right|\geq\k$. If $h+l=0$,

$$
    \left|1- q^l\frac{\f}{\e}\right|  
     \geq\min\lb \left|1- \frac{s}{t}\right|,  \left|1- \frac{t}{s}\right|\rb \geq\const\times(t-s).
$$

    By a similar argument, in the other cases when
    $$\frac{\f}{\e}\in\left\{\frac{\sqrt{s}}{At},\frac{1}{A^2t},\frac{C}{At},\frac{At}{C}\right\}\times\{q^r:r\in\ZZ\},$$
 then $\left|1-q^l\frac{\f}{\e}\right|$ can be lower bounded by $\const\times(t-s)$ if 
$$q^{-l}\in\left\{\frac{1}{A},\frac{1}{A^2},\frac{C}{A},\frac{A}{C}\right\}\quad\text{and}\quad \frac{\f}{\e}= \left\{\frac{\sqrt{s}}{t},\frac{1}{t},t\right\}\times q^{-l};$$ and for  all other  $l\in\ZZ$, $\left|1-q^l\frac{\f}{\e}\right|$ can be lower bounded by $\k$. The above fact follows from  
$$\min\lb \left|1-\frac{\sqrt{s}}{t}\right|,\left|1-\frac{1}{t}\right|,|1-t|\rb\geq  \const\times(t-s),$$
and from our choices of $\ep>0$ and $\k>0$.

    \item [Case 3.] When $\e=c$, $\f=d$ and $\f/\e\leq q^{2k}s/t$. In this case $q^l\f/\e  \leq q^{2k+l}s/t<q^{2k+l}$. If either $k\geq 1$ or $k=0$ and $l\geq 1$, then $\f/\e<q$ and $|1-\f/\e|>1-q$. If $k=l=0$ then $\f/\e\leq s/t$ and
    $$ \left|1-\frac{\f}{\e}\right|\geq1- \frac{s}{t}\geq\const\times(t-s).$$
\end{enumerate}
In view of Case 1, Case 2 and Case 3 above, we conclude the proof of Claim \ref{claim:only need}.
\end{proof}

We return to the proof. By the reasoning below Claim \ref{claim:only need}, we conclude the proof of \eqref{eq:only need appendix B}.  
 It is clear from our choice of $\ep>0$ that it is independent of $D$.
Our last goal is to show that the  constant $K$ can also be chosen as not depending on $D$. Since $D\in(-1,0]$, $\max(|a|,|b|,|c|,|d|)$ can be bounded by a constant not depending on $D$. Therefore the total number of atoms $\card\lb F(a,b,c,d)\rb$ as well as the constant prefactors appearing in the atom mass bounds \eqref{eq:in proof bound atom masses} can be bounded by  constants independent of $D$. In the proof of Claim \ref{claim:only need} (which bound all the linear factors $|1-q^l\f/\e|$), $D$ does not play a role at all. Therefore the uniform constants chosen from Claim \ref{claim:only need} does not depend on $D$.
 In summary,  constant $K$ can be chosen as independent of $D$. We conclude the proof.
\end{proof}
 
\begin{proposition}\label{prop:bound on pi1}
    Assume $A,C\geq0$, $\max(A,C)>1$, $-1<B,D\leq 0$, $q\in[0,1)$ and $|(ABCD)_{\infty}|\leq1$. Assume also  $A/C\notin\{q^l:l\in\ZZ\}$ if $A,C\geq 1$. Then there exists a positive constant $L$ depending on $A,C$ and $q$ such that $$\lnor\pi_1\rnor\leq \frac{L}{1-BD}\lv\pi_1\lb\{y_0(1)\}\rb\rv.$$
\end{proposition}
\begin{proof}
    We will prove the result for the high density phase $A>1$, $A>C$. The result for the low density phase follows from symmetry. The atoms are generated by $A$ and possibly also by $C$. 
    In view of Lemma \ref{lem:bound TV} 
   we want to bound the mass $\lv\pi_1\lb\{y_0(1)\}\rb\rv$ of first atom from below and masses of all other atoms from above. 
    We have:  
     $$\pi_1\lb\{y_0(1)\}\rb=p_0^{\aa}(A,B,C,D)=\frac{(A^{-2},BC,BD,CD)_\infty}{(B/A,C/A,D/A,ABCD)_\infty}. $$
    Notice that on $x\in(-\infty,0]$, $(x)_{\infty}\geq1$ is a decreasing function. Using the assumption $-1<B,D\leq 0$, one can lower bound this atom mass:
    \be\label{eq:lower}|\pi_1\lb\{y_0(1)\}\rb|\geq P \frac{1-BD}{|(ABCD)_{\infty}|},\ee
    where $P$ is a positive constant depending only on $A,C$ and $q$. 
    
    We then look at the masses of other (possible) atoms generated by $\aa$: For $k\geq1$ and $Aq^k\geq1$,
    \begin{multline*}
        p_k^\aa(A,B,C,D) \\
        =\frac{(A^{-2},BC,BD,CD)_\infty}{(B/A,C/A,D/A,ABCD)_\infty}
    \frac{q^k(1-A^2q^{2k})(A^2,AB,AC,AD)_k}{(q)_k(1-A^2)A^{4k}\prod_{l=1}^{k}\lb(B/A-q^l)(C/A-q^l)(D/A-q^l)\rb},
    \end{multline*} 
    along with the masses of the (possible) atoms generated by $\cc$: For $k\geq0$ and $Cq^k\geq1$,
    \begin{multline*}
        p_k^\cc(A,B,C,D) \\
        =\frac{(C^{-2},AB,BD,AD)_\infty}{(B/C,A/C,D/C,ABCD)_\infty}
    \frac{q^k(1-C^2q^{2k})(C^2,BC,AC,CD)_k}{(q)_k(1-C^2)C^{4k}\prod_{l=1}^{k}\lb(B/C-q^l)(A/C-q^l)(D/C-q^l)\rb}.
    \end{multline*} 
    Again, using the assumption $-1<B,D\leq 0$, one can upper bound these atom masses:
    \be\label{eq:upper}|p_k^{\eee}(A,B,C,D)|\leq\frac{Q }{|(ABCD)_{\infty}|},\ee
    for any atom $y_k^{\eee}\in F(A,B,C,D)$,
    where $Q$ is a uniform positive constant depending   on $A,C$ and $q$.

In view of Lemma \ref{lem:bound TV}, the lower bound \eqref{eq:lower}, the upper bound \eqref{eq:upper} and our assumption that $|(ABCD)_{\infty}|\leq1$, we conclude the proof.\end{proof}

\section{Time-reversal of Askey--Wilson signed measures}
\label{sec:time reversal}

We prove a special symmetry of multi-dimensional Askey--Wilson signed measures known as the time-reversal. 
In the case that this symmetry only involves actual probability measures (i.e., on the fan region $AC<1$), this symmetry appears as \cite[equation (5.10)]{bryc2017asymmetric}. However the proof is not explained clearly enough therein. Here we adopt a different approach to prove the result for the general case.
\begin{proposition}\label{prop:time reversal}
  Assume $q\in[0,1)$ and $A,B,C,D\in\RR$ satisfying $q^lABCD\neq1$ for all $l\in\NN$.    We have:
    \be 
    \pi_{t_1,\dots,t_m}^{(A,B,C,D)}\lb\d x_1,\dots,\d x_m\rb=
    \pi_{1/t_m,\dots,1/t_1}^{(C,D,A,B)}\lb\d x_m,\dots,\d x_1\rb
    \ee
    for any $0<t_1\leq\dots\leq t_m$ for which the multi-dimensional Askey--Wilson signed measures on both sides of this identity are well-defined. 
\end{proposition}
\begin{remark}
    In view of the definition \eqref{eq:definition of multi-time} of $\pi_{t_1,\dots,t_m}^{(A,B,C,D)}\lb\d x_1,\dots,\d x_m\rb$, it is well-defined when $\pi^{(A,B,C,D)}_{t_1}(\d y)$ and each $P^{(A,B)}_{t_{i-1},t_i}(x,\d y)$, $i=2,\dots,m$ for any $x\in U_{t_{i-1}}$ are well-defined Askey--Wilson signed measures, i.e. their quadruples of entries lie in the subset $\Omega\subset\CC^4$ in Definition \ref{def:AW}.
\end{remark}

Before we begin the proof of Proposition \ref{prop:time reversal}, we recall  the orthogonality and the projection formula of Askey--Wilson polynomials and signed measures. The Askey--Wilson polynomials $w_j(x):=w_j(x;a,b,c,d)$, $j\in \NN$ are   defined by three term recurrence  in \cite{askey1985some} (see also \cite[Section 2.1]{wang2023askey}).  
\begin{lemma}[Corollary  2.6 and Theorem 2.7  in \cite{wang2023askey}]\label{lem:orthogonality}
    Assume $a,b,c,d\in\Omega$. Then for any $j,k\in\NN$,
\be \label{eq:orthogonality on R AW}
\int_\mathbb{R}\nu(\d x;a,b,c,d)w_j(x)w_k(x)=\delta_{jk}\frac{(1-q^{j-1}abcd)(q,ab,ac,ad,bc,bd,cd)_j}{(1-q^{2j-1}abcd)(abcd)_j}.
\ee 
\end{lemma}
\begin{lemma}[Proposition 3.3 in \cite{wang2023askey}]\label{lem:projection formula}
  Assume $A,B,C,D\in\RR$. For any $s\leq t$ such that $P_{s,t}(x,\d y)$ is well-defined for all $x\in U_s$, we have:
  \be\label{eq:projection formula}\int_\mathbb{R}p_j(y;t) P_{s,t}(x,\d y)=p_j(x;s),\ee
where for $j\in\NN$, $$p_j(x;t):=t^{j/2}(ABt)_j^{-1}w_j\lb x;A\sqrt{t},B\sqrt{t},C/\sqrt{t},D/\sqrt{t}\rb.$$  
\end{lemma}
\begin{remark}
    We note that Lemma \ref{lem:orthogonality} and Lemma \ref{lem:projection formula} above are originally due to \cite{askey1985some,bryc2010askey} and later generalized by \cite{wang2023askey} to the cases of signed measures. 
\end{remark}

 We now start to prove Proposition \ref{prop:time reversal}.
\begin{proof}[Proof of Proposition \ref{prop:time reversal}]
    We first prove the result for $m=1$ and $m=2$, then we use induction to prove it for $m\geq 3$.

    For $m=1$, we recall that the Askey--Wilson signed measures are symmetric with respect to their entries. This fact in particular follows from the Askey--Wilson polynomials $w_j(x)=w_j(x;a,b,c,d)$ being symmetric with respect to parameters $a,b,c$ and $d$. Therefore we have:
    \begin{multline*}
        \pi_t^{(A,B,C,D)}(\d x)=\nu\lb \d x;A\sqrt{t},B\sqrt{t},\frac{C}{\sqrt{t}},\frac{D}{\sqrt{t}}\rb\\
    =\nu\lb \d x;\frac{C}{\sqrt{t}},\frac{D}{\sqrt{t}},A\sqrt{t},B\sqrt{t}\rb=\pi_{1/t}^{(C,D,A,B)}(\d x).
    \end{multline*} 
    
    For $m=2$, we compute, for $s<t$ and $j,k\in\mathbb{N}_0$:
    \begin{equation}\label{eq:big}
    \begin{split}
        &\int_{\R^2}\pi_{s,t}^{(A,B,C,D)}\lb\d x,\d y\rb w_j\lb y;A\sqrt{t},B\sqrt{t},\frac{C}{\sqrt{t}},\frac{D}{\sqrt{t}}\rb
        w_k\lb x;A\sqrt{s},B\sqrt{s},\frac{C}{\sqrt{s}},\frac{D}{\sqrt{s}}\rb\\
        &=\frac{(ABt)_j(ABs)_k}{t^{j/2}s^{k/2}}\int_{\R^2}\pi_{s,t}^{(A,B,C,D)}\lb\d x,\d y\rb p_j(y;t)p_k(x;s)\\
        &=\frac{(ABt)_j(ABs)_k}{t^{j/2}s^{k/2}}\int_{\R}\pi_s^{(A,B,C,D)}(\d x)p_k(x;s)\int_{\R}P^{(A,B)}_{s,t}(x,\d y)p_j(y;t)\\
        &=\frac{(ABt)_j(ABs)_k}{t^{j/2}s^{k/2}}\int_{\R}\pi_s^{(A,B,C,D)}(\d x)p_k(x;s) p_j(x;s)\\
        &=\frac{(ABt)_js^{j/2}}{(ABs)_jt^{j/2}}\int_{\R}\pi_s^{(A,B,C,D)}(\d x)w_j\lb x;A\sqrt{s},B\sqrt{s},\frac{C}{\sqrt{s}},\frac{D}{\sqrt{s}}\rb w_k\lb x;A\sqrt{s},B\sqrt{s},\frac{C}{\sqrt{s}},\frac{D}{\sqrt{s}}\rb\\
        &=\frac{(ABt)_js^{j/2}}{(ABs)_jt^{j/2}}\delta_{jk}\frac{(1-q^{j-1}ABCD)(q,ABs,AC,AD,BC,BD,CD/s)_j}{(1-q^{2j-1}ABCD)(ABCD)_j}\\
        &=\delta_{jk}\lb\frac{s}{t}\rb^{j/2}(ABt)_j\lb\frac{CD}{s}\rb_j\frac{(1-q^{j-1}ABCD)(q,AC,AD,BC,BD)_j}{(1-q^{2j-1}ABCD)(ABCD)_j}.
    \end{split}
    \end{equation} 
We have used the projection formula (Lemma \ref{lem:projection formula}) in the third step and the orthogonality property (Lemma \ref{lem:orthogonality}) in the second to last step. 

The RHS of the above identity \eqref{eq:big} remains the same when we swap $x\leftrightarrow y$, $j\leftrightarrow k$, $A\leftrightarrow C$, $B\leftrightarrow D$  and also take $s\mapsto 1/t$, $t\mapsto 1/s$. Therefore we have:  
    \begin{equation*}
        \begin{split}
            &\int_{\R^2}\pi_{s,t}^{(A,B,C,D)}\lb\d x,\d y\rb w_j\lb y;A\sqrt{t},B\sqrt{t},\frac{C}{\sqrt{t}},\frac{D}{\sqrt{t}}\rb w_k\lb x;A\sqrt{s},B\sqrt{s},\frac{C}{\sqrt{s}},\frac{D}{\sqrt{s}}\rb\\
        &=\delta_{jk}\lb\frac{s}{t}\rb^{j/2}(ABt)_j\lb\frac{CD}{s}\rb_j\frac{(1-q^{j-1}ABCD)(q,AC,AD,BC,BD)_j}{(1-q^{2j-1}ABCD)(ABCD)_j}\\
        &=\int_{\R^2}\pi_{1/t,1/s}^{(C,D,A,B)}\lb\d y,\d x\rb 
        w_k\lb x;\frac{C}{\sqrt{s}},\frac{D}{\sqrt{s}},A\sqrt{s},B\sqrt{s}\rb
        w_j\lb y;\frac{C}{\sqrt{t}},\frac{D}{\sqrt{t}},A\sqrt{t},B\sqrt{t}\rb  \\
        &=\int_{\R^2}\pi_{1/t,1/s}^{(C,D,A,B)}\lb\d y,\d x\rb w_j\lb y;A\sqrt{t},B\sqrt{t},\frac{C}{\sqrt{t}},\frac{D}{\sqrt{t}}\rb w_k\lb x;A\sqrt{s},B\sqrt{s},\frac{C}{\sqrt{s}},\frac{D}{\sqrt{s}}\rb.
        \end{split}
    \end{equation*}
We notice that $w_j(x)$ has degree $j$ for $j\in\NN$. One can use induction within the above identity to conclude that, for every $j,k\in\mathbb{N}_0$,
$$\int_{\R^2}\pi_{s,t}^{(A,B,C,D)}\lb\d x,\d y\rb x^ky^j=\int_{\R^2}\pi_{1/t,1/s}^{(C,D,A,B)}\lb\d y,\d x\rb x^ky^j.$$
We next choose a large enough closed ball $K\subset\RR^2$ whose interior contains the supports of both of the signed measures $\pi_{s,t}^{(A,B,C,D)}\lb\d x,\d y\rb$ and $\pi_{1/t,1/s}^{(C,D,A,B)}\lb\d y,\d x\rb$. 
By the Stone--Weierstrass theorem, any continuous function $g$ on  $K$ can be uniformly approximated by polynomials. Therefore for any continuous function $g$ on $\R^2$, we have:
\be\label{eq:moment proof}\int_{\R^2}\pi_{s,t}^{(A,B,C,D)}\lb\d x,\d y\rb g(x,y)=\int_{\R^2}\pi_{1/t,1/s}^{(C,D,A,B)}\lb\d y,\d x\rb g(x,y).\ee
Hence both of the signed measures  define the same bounded linear functional  $$C(K)\longrightarrow\RR.$$
In view of (the uniqueness part of) the Riesz representation theorem (see for example \cite[Theorem 6.19]{rudin1987real}), we have
$$\pi_{s,t}^{(A,B,C,D)}\lb\d x,\d y\rb=\pi_{1/t,1/s}^{(C,D,A,B)}\lb\d y,\d x\rb.$$

We conclude the proof  for $m=2$.

Assume that the result holds for some $m\geq 2$, we prove it for the case $m+1$. We have:
\begin{equation*}
    \begin{split}
        &\pi^{(A,B,C,D)}_{t_1,\dots,t_{m+1}}(\d x_1,\dots,\d x_{m+1})\\ 
        &=\pi^{(A,B,C,D)}_{t_1,\dots,t_m}(\d x_1,\dots,\d x_m)P^{(A,B)}_{t_m,t_{m+1}}(x_m,\d x_{m+1})\\
        &=\pi^{(C,D,A,B)}_{1/t_m,\dots,1/t_1}(\d x_m,\dots,\d x_1)P^{(A,B)}_{t_m,t_{m+1}}(x_m,\d x_{m+1})\\
        &=\pi^{(C,D,A,B)}_{1/t_m}(\d x_m)P^{(C,D)}_{1/t_m,1/t_{m-1}}(x_m,\d x_{m-1})\dots P^{(C,D)}_{1/t_2,1/t_1}(x_2,\d x_1) P^{(A,B)}_{t_m,t_{m+1}}(x_m,\d x_{m+1})\\
        &=\pi^{(A,B,C,D)}_{t_m}(\d x_m) P^{(A,B)}_{t_m,t_{m+1}}(x_m,\d x_{m+1}) P^{(C,D)}_{1/t_m,1/t_{m-1}}(x_m,\d x_{m-1})\dots P^{(C,D)}_{1/t_2,1/t_1}(x_2,\d x_1) \\
        &=\pi^{(A,B,C,D)}_{t_m,t_{m+1}}(\d x_m,\d x_{m+1}) P^{(C,D)}_{1/t_m,1/t_{m-1}}(x_m,\d x_{m-1})\dots P^{(C,D)}_{1/t_2,1/t_1}(x_2,\d x_1) \\
        &=\pi^{(C,D,A,B)}_{1/t_{m+1},1/t_m}(\d x_{m+1},\d x_m) P^{(C,D)}_{1/t_m,1/t_{m-1}}(x_m,\d x_{m-1})\dots P^{(C,D)}_{1/t_2,1/t_1}(x_2,\d x_1) \\
        &=\pi^{(C,D,A,B)}_{1/t_{m+1},\dots,1/t_1}(\d x_{m+1},\dots,\d x_1), 
    \end{split}
\end{equation*}
where we have used the result for $m$, $1$ and $2$.
By induction, we conclude the proof. 
\end{proof}

\end{appendix}

\end{document}